\documentclass{amsart}
\usepackage[utf8]{inputenc}  
\usepackage{amsmath}
\usepackage{comment}
\usepackage{amsthm}
\usepackage{amssymb}
\usepackage{hyperref}
\usepackage{mathtools}
\mathtoolsset{mathic}
\hypersetup{
  colorlinks,
  linkcolor={red!70!black},
  citecolor={blue!70!black},
  urlcolor={blue!90!black}
  }
\usepackage{esint}
\usepackage{fourier}

\newcommand{\paramY}{\langle Y\rangle}

\newcommand{\cC}{\mathcal{C}}
\newcommand{\cD}{\mathcal{D}}
\newcommand{\cDv}{{\mathcal{D}_{\mathsf{v}}}}

\newcommand{\cH}{\mathcal{H}}

\newcommand{\cV}{\mathcal{V}}

\newcommand{\HH}{\mathbb{H}}

\newcommand{\Kor}{\mathrm{Kor}}
\DeclareMathOperator{\Cone}{Cone}
\DeclareMathOperator{\Lip}{Lip}

\newcommand{\cVv}{{\mathcal{V}_{\mathsf{v}}}}
\newcommand{\cVh}{{\mathcal{V}_{\mathsf{h}}}}
\newcommand{\Aff}{\mathsf{Aff}}

\usepackage{pgfornament}

\DeclareMathOperator{\slope}{slope}

\DeclareMathOperator{\supp}{supp}
\DeclareMathOperator{\inter}{int}

\renewcommand{\H}{\mathbb{H}}
\newcommand{\R}{\mathbb{R}}
\newcommand{\Z}{\mathbb{Z}}
\newcommand{\N}{\mathbb{N}}
\newcommand{\from}{\colon}

\newcommand{\zero}{\mathbf{0}}
\newcommand{\ud}[0]{\,\mathrm{d}}

\newtheorem{thm}{Theorem}[section]
\newtheorem{lemma}[thm]{Lemma}
\newtheorem{cor}[thm]{Corollary}
\newtheorem{prop}[thm]{Proposition}
\theoremstyle{remark}

\newtheorem{defn}[thm]{Definition}

\title[The strong geometric lemma in $\H_1$]{The strong geometric lemma in the Heisenberg group}

\author{Vasileios Chousionis}
\address{Department of Mathematics, University of Connecticut}
\email{vasileios.chousionis@uconn.edu}

\author{Sean Li}
\address{Department of Mathematics, University of Connecticut}
\email{sean.li@uconn.edu}

\author{Robert Young}
\address{Courant Institute of Mathematical Sciences, New York University}
\email{ryoung@cims.nyu.edu}

\thanks{V.~C.\ was supported by Simons Foundation Collaboration grant 521845. R.~Y.\ was supported by NSF grant 2005609}
\date{\today}


\begin{document}

\maketitle
\begin{abstract}
We prove that in the first Heisenberg group, unlike  Euclidean spaces and higher dimensional Heisenberg groups, the best possible exponent for the strong geometric lemma for intrinsic Lipschitz graphs is $4$ instead of $2$.  Combined with earlier work from \cite{CLY1} and \cite{CLY2}, our result completes the proof of the strong geometric lemma in Heisenberg groups. One key tool in our proof, and possibly of independent interest, is a suitable refinement of the foliated coronizations which first appeared in \cite{NY-Foliated}.
\end{abstract}

\section{Introduction}

According to Rademacher's theorem, Lipschitz surfaces in $\R^n$ infinitesimally resemble planes. However, for some more global questions, like the boundedness of singular integrals, the information we get from Rademacher's theorem is too qualitative. To answer these questions, we need to know that Lipschitz graphs can be effectively approximated by affine planes ``at most places and scales.'' The traditional way to quantify such a statement is via the notion of $\beta$--numbers, introduced by Jones in \cite{JonesSquare, JonesTSP}. Since Jones's work, $\beta$--numbers and several variants have played an important role in geometric harmonic analysis and geometric measure theory in $\R^n$ as well as in other spaces \cite{Oki, SchulTSP, LiSchul, LiSchul2, FFP, CLZ,Li-strat, DavidSchul, CLY1, CLY2}.  

Given  $E \subset \R^n$ and a ball $B(x,r)$, the $m$-dimensional $\beta$--number is defined as
\begin{equation*}
 \beta_{E}(x,r)=\inf_{L}  r^{-m} \int_{B(x,r)\cap E} \frac{d(y,L)}{r}  \; \ud  \mathcal{H}^{m}(y)
\end{equation*}
where $L$ is taken over all $m$--dimensional planes and $\mathcal{H}^m$ denotes the $m$--dimensional Hausdorff measure. Thus, $\beta_{E}(x,r)$ is an measures how flat the is the set $E \cap B(x,r)$ on average. This function is scale-invariant in the sense that $\beta_{t E}(t x,t r)=\beta_{ E}(x, r)$ for all $t>0$.
By a fundamental result of Dorronsoro \cite{dor}, if $E$ is an $m$--dimensional Lipschitz graph, then the squares of the $\beta$--numbers satisfy the following Carleson packing condition: there is a $c>0$ depending only on the Lipschitz constant such that
\begin{equation}
\label{do}
\int_0^R \int_{E \cap B(x,R) } \beta_{E} (y,r)^2 \ud \mathcal{H}^m (y) \frac{dr}{r} \le c R^m, \quad \mbox{ for }R>0, x\in E.
\end{equation}
This is known as the \textit{strong geometric lemma} with exponent $2$.

The strong geometric lemma for Lipschitz graphs has been one of the cornerstones in the theory of uniform rectifiability, which was originally developed by David and Semmes in the early 90s. In particular, an Ahlfors $m$--regular subset of $\R^n$ is called uniformly $m$--rectifiable if and only if it satisfies \eqref{do} for some $c>0$. David and Semmes in \cite{DS1,DS2} introduced an extensive suite of geometric and analytic notions in order to provide various characterizations of uniformly rectifiable sets in Euclidean spaces. They thus provided a rich and influential geometric foundation for the study of singular integral operators on lower dimensional subsets of $\R^n$.

In the last 20 years there have been systematic efforts toward the development of geometric measure theory on sub-Riemannian spaces, see e.g.~the lecture notes \cite{SCnotes,perttirev} for some recent surveys. In particular, this research program has grown substantially on Carnot groups, a class of nilpotent Lie groups that admit dilations, i.e., maps that scale the metric by a constant. Indeed, by a result of LeDonne \cite{LD}, Carnot groups are the \textit{only} locally compact geodesic spaces which admit dilations and are isometrically homogeneous. This makes these groups ideal environments to study geometric and analytic questions involving many different scales. The simplest examples of nonabelian Carnot groups are the $(2n+1)$--dimensional Heisenberg groups $\H_n$. 

To introduce meaningful notions of rectifiability in $\H_n$ (or any other Carnot group), one should first come up with the analogue of a Lipschitz graph in that setting. This is rather subtle because, for example, Heisenberg groups  cannot be viewed as Cartesian products of subgroups. One might consider Lipschitz images $\R^{k} \to \H_n$, but this approach only works when $k\le n$. Ambrosio and Kirchheim \cite{AK} proved that any Lipschitz image $f(\R^{k}) \subset \H_n$ has vanishing $\mathcal{H}^k$--measure for $k > n$.  Therefore one should introduce a notion of {\em intrinsic} graphs fitting the sub-Riemannian group structure. This was achieved by Franchi, Serapioni, and Serra Cassano in  \cite{FSS1}. Similarly to Euclidean Lipschitz graphs, intrinsic Lipschitz graphs in $\H_n$ satisfy a cone condition, which will be defined in the next section. 

Intrinsic Lipschitz graphs feature prominently in the newly emerging theories of low-codimensional rectifiability \cite{MSS,FSSCDiff,VitIntLip,SCnotes}. One reason for this is that, like Lipschitz graphs in $\R^n$, intrinsic Lipschitz graphs in $\H_n$ satisfy a version of Rade\-macher's theorem; they infinitesimally resemble planes almost everywhere \cite{FSSCDiff}. As in $\R^n$, answering questions about singular integrals and uniform rectifiability in $\H_n$ requires more quantitative bounds, which has led to the study of various notions of quantitative rectifiability in intrinsic Lipschitz graphs \cite{CFO1,NY1,NY-Foliated,FORig,Rig, CLY1, CLY2}.

Broadly speaking, the aforementioned works seek to obtain quantitative bounds on how intrinsic Lipschitz graphs can be approximated by vertical planes and vertical sets. As in the Euclidean case, these works are partly motivated by a singular integral operator which can be viewed as the natural analogue of the the (Euclidean) $1$--codimensional Riesz transform. Notably, and as in $\R^n$, this new singular integral is related to removability for Lipschitz harmonic functions in $\H_n$, see \cite{CFO2,FORiesz,CLY2}. 

One of the key objects in the study of uniform rectifiability in Heisenberg groups is the following  codimension--1 version of the $\beta$--numbers. For $E \subset \H_n, x\in \H_n,$ and $r>0$, we define
\begin{equation}\label{eq:def beta} \beta_{E}(x,r)=\inf_{L \in \mathsf{VP}}  r^{-2n-1} \int_{B(x,r)\cap E} \frac{d(y,L)}{r} \; \ud  \mathcal{H}^{2n+1}(y)
\end{equation}
where $\mathsf{VP}$ denotes the set 
of codimension–1 planes which are parallel to the z-axis. Note that the quantities  $\beta_{E}(x,r)$ are scale-invariant  since $\beta_{\delta_t(E)}(\delta_t(x),tr)=\beta_{E}(x,r)$ for any $t>0$. 
We record that an $L_\infty$ version of the codimension--1 $\beta$--numbers in the Heisenberg group was introduced in \cite{CFO1}, where a weaker qualitative 
(Heisenberg) analogue of \eqref{do}, known as the
weak geometric lemma, was obtained for intrinsic Lipschitz graphs in $\H_n$.  

In \cite{CLY1} we established the direct analogue of the strong geometric lemma with exponent $2$ for intrinsic Lipschitz graphs in $\H_n$ for $n \geq 2$.
\begin{thm}\label{thm:beta2 highdim}\cite{CLY1}
  Let $n\ge 2$ and let $\Gamma$ be an intrinsic $L$--Lipschitz graph in $\H_n$.  Then, for any $y \in \Gamma$ and any $R>0$,
  \begin{align}
  \int_0^R \int_{B(y,R) \cap \Gamma} \beta_{\Gamma}(x,r)^2 \ud \mathcal{H}^{2n+1}(x) \frac{\ud r}{r} \lesssim_{L} R^{2n+1}. \label{eq:betahighdim}
  \end{align}
\end{thm}
However, in \cite{CLY2} we showed that unlike the Euclidean case, the strong geometric lemma fails in $\H_1$ for all exponents $s \in [2,4)$.
\begin{thm}\cite{CLY2} \label{thm:beta-theorem}There exist a constant $L>0$, a radius $R>0$, and a sequence of $L$--intrinsic Lipschitz graphs $(\Gamma_n)_{n \in \N}$ in $\H_1$ such that $\zero \in \Gamma_n$ for all $n$ and
  \begin{equation}\label{eq:unbdbeta}\lim_{n \to \infty} \int_{0}^R\int_{B(0,R) \cap \Gamma_n}  \beta_{\Gamma_n}(x,r)^s \ud \mathcal{H}^{3}(x) \frac{\ud r}{r}=+\infty 
  \end{equation}
 for all $s \in [2,4)$.
\end{thm}

In this paper we complete this line of research by showing that the strong geometric lemma holds in $\H_1$ with exponent $4$. Our main theorem is the following:
\begin{thm}\label{thm:beta4}
 Let \(\Gamma\) be an intrinsic \(L\)--Lipschitz graph in $\H_1$.  Then for any $y\in \Gamma$ and any $R>0$,
  \begin{equation}\label{eq:beta4-main}
    \int_{B(y,R) \cap \Gamma} \int_0^R \beta_{\Gamma}(x,r)^4\;\frac{\mathrm{d} r}{r}\ud \cH^3(x) \lesssim_L R^3.
  \end{equation}
\end{thm}

We prove Theorem~\ref{thm:beta4} by using the foliated corona decompositions constructed in \cite{NY-Foliated}. These decompose an intrinsic Lipschitz graph into rectangular regions called pseudoquads, whose aspect ratio depends on the shape of the corresponding intrinsic Lipschitz graph. At points and scales where the graph is flat, the pseudoquads are short and wide (large aspect ratio), and at points where the graph is bumpy, the pseudoquads are tall and skinny (small aspect ratio). 

The decomposition satisfies a weighted Carleson condition, which bounds the total size of the pseudoquads in the decomposition, weighted by the inverse fourth power of the aspect ratio (see Definition~\ref{def:weighted Carleson}). This parallels the construction of bumpy surfaces in \cite{NY-Foliated} and \cite{CLY2}, where Theorem~\ref{thm:beta-theorem} was proved by constructing surfaces with bumps with large aspect ratio. Each bump increases the area of the surface in proportion to the inverse fourth power of its aspect ratio, which leads to the exponent $4$ in Theorem~\ref{thm:beta-theorem}. The proof of Theorem~\ref{thm:beta4} relates the geometry of a foliated corona decomposition to the $\beta$--numbers of the associated graph and uses the weighted Carleson condition to prove the inequality \eqref{eq:beta4-main}.

One can generalize some of these questions by defining analogues of $\beta$--numbers for other subsets and other Carnot groups. For instance, quantitative rectifiability of $1$--dimensional subsets of Carnot groups has been studied in \cite{Li-strat, CLZ}, where relative beta numbers where introduced in order to study the travelling salesman problem and its connections to singular integrals. Little is known, however, about quantitative rectifiability for higher-dimensional subsets, even surfaces in $\H_n$ with topological dimension between $2$ and $2n-1$.

Similarly, though one can define intrinsic Lipschitz graphs in arbitrary Carnot groups (see \cite[Section 4.5]{SCnotes}), and these objects play an important role in rectifiability, the quantitative rectifiability of these graphs has not been systematically studied. For instance, given $p\in \N$, is there a Carnot group $G$ such that intrinsic Lipschitz graphs in $G$ satisfy an analogue of \eqref{eq:beta4-main} only for exponents greater than or equal to $p$? 

\subsection{Outline of paper}
In Section~\ref{sec:prelim}, we give some basic definitions and theorems that we will use in the rest of the paper. We also recap the results of \cite{NY-Foliated} on foliated corona decompositions. In Section~\ref{sec:initial}, we prepare to prove Theorem~\ref{thm:beta4} by reducing \eqref{eq:beta4-main} to an inequality involving the intrinsic Lipschitz function $\psi$ that parametrizes $\Gamma$.

In order to use foliated corona decompositions to prove Theorem~\ref{thm:beta4}, we must improve some of the bounds from \cite{NY-Foliated}. To achieve this, we first (in Section~\ref{sec:coronizations}) introduce foliated coronizations and paramonotone-stopped foliated coronizations (PSFCs), which are decompositions with some convenient additional bounds. Then, in Section~\ref{sec:lp-approx}, we prove a bound on the $L_4$ distance between a pseudoquad in an intrinsic Lipschitz graph and a plane. This improves the corresponding $L_1$ bound in \cite{NY-Foliated}. 
Finally, in Section~\ref{sec:proof}, we prove Theorem~\ref{thm:beta4}.

\section{Preliminaries}\label{sec:prelim}

\subsection{The Heisenberg group} The Heisenberg group $\HH:=\H_1$ is the nonabelian Lie group whose elements are points in $\R^3$ and whose group operation is given by
\begin{align}\label{eq:heis-mult}
  (x,y,z) (x',y',z') = \left(x+x',y+y',z+z' + \frac{xy' - x'y}{2} \right).
\end{align}
The identity element is $\zero:=(0,0,0)$ and the inverse of $v=(x,y,z)$ is the element $v^{-1}=(-x,-y,-z)$. 
Let $X=(1,0,0),Y=(0,1,0),Z=(0,0,1)\in \HH$ and let $x,y,z \from \HH\to \R$ be the coordinate functions. Given any $p\in \HH$, $p\ne 0$, we will denote by $\langle p\rangle$  the one-parameter subgroup containing $p$; in these coordinates, $\langle p \rangle$ is the subspace spanned by $p$. This lets us write $w^t = t w$ for $w \in \HH$ and $t \in \R$; when $ t\in \Z$, this agrees with the usual notion of exponentiation.

The center of the group is $\langle Z \rangle = \{(0,0,z) \mid z \in \R\}$.  An element $p \in \HH$  such that  $z(p) = 0$ is called a {\it horizontal vector}, and we denote by $A$ the set of horizontal vectors. Let $\pi\from \HH\to \R^2$ be the projection $\pi(x,y,z) = (x,y)$.

The \emph{Kor\'anyi metric} on $\HH$ is the left-invariant metric defined by
$$d_{\Kor}(p,p'):=\|p^{-1} p'\|_{\Kor},$$
where
$$\|(x,y,z)\|_{\Kor}:=\sqrt[4]{(x^2+y^2)^2 + 16 z^2}.$$
We note that the family of automorphisms:
$s_t \from  \HH \to \HH, t \in \R,$ 
\begin{align*}
 s_t(x,y,z) = (t x, t y, t^2 z),
\end{align*}
dilate the Kor\'anyi metric metric; for $t \geq 0$ and $p, q \in \HH$,
\begin{align*}
d_{\Kor}(s_t (p), s_t(q))= t d_{\Kor} (p,q).
\end{align*}

Let $I \subset \R$ be an open interval. The map $\gamma\from I \to \HH$ is a {\it horizontal curve} if $x \circ \gamma, y \circ \gamma, z \circ \gamma \from I \to \R$ are Lipschitz (and thus $\gamma'$ is defined almost everywhere on $I$) and $$\frac{\ud}{ \ud s} \left[\gamma (t)^{-1} \gamma(s)\right]\big|_{s=t} \in A,$$
for almost every $t \in I$.  Cosets of the form $L=v \langle aX+bY \rangle$  where $(a,b) \in \R^2 \setminus \{(0,0)\}$ and $v \in \HH$, are called {\it horizontal lines}. The \emph{slope} of a horizontal line $L$ is the slope of its projection $\pi(L)$, i.e., if $L=v \langle aX+bY \rangle$, then $\slope L = \frac{b}{a}$ when $a\ne 0$ and $\slope L = \infty$ when $a=0$. 

A plane parallel to the $z$--axis is called a \emph{vertical plane}. Note that if $V$ is a vertical plane, the projection $\pi(V)$ is a line and we define the slope of $V$ as $\slope(V) := \slope(\pi(V))$. We will frequently use the vertical plane $V_0 := \{y = 0\}$. 

\subsection{Intrinsic Lipschitz graphs and characteristic curves}

Intrinsic graphs and intrinsic Lipschitz graphs are classes of surfaces in $\HH$ that play an important role in the study of rectifiability. Similarly to \cite{CLY2} we will define intrinsic graphs in terms of functions from $\HH$ to $\R$ that are constant along cosets of $\paramY$. In particular, if $\phi\from \HH \to \R$ is constant on cosets of $\paramY$, the \emph{intrinsic graph} of $\phi$ is the set
$$\Gamma_\phi=\{vY^{\phi(v)}\mid v\in V_0\}=\{p\in \HH \mid \phi(p)=y(p)\}.$$
(Many authors call these \emph{entire intrinsic graphs} and use ``intrinsic graph'' to refer to closed subsets of $\Gamma_\phi$.)

Note that left-translations and dilations of intrinsic graphs are also intrinsic graphs.  We parametrize $\Gamma_\phi$ by the map $\Psi_\phi\from V_0 \to \Gamma_\phi$,
$$\Psi_\phi(p)=pY^{\phi(p)-y(p)}.$$ 
This map projects $V_0$ to $\Gamma_\phi$ along cosets of $\paramY$, and if $\phi$ is continuous, it is a homeomorphism from $V_0$ to $\Gamma_\phi$.

Conversely, we let $\Pi \from \HH \to V_0$ be the (nonlinear) projection along cosets of $\paramY$ given by $\Pi(v)=v Y^{-y(v)}, v \in \HH$. Equivalently,
$$\Pi(x,y,z)=\left(x,0,z-\frac{1}{2} xy\right).$$
Although $\Pi$ is not a homomorphism, it commutes with the dilations $s_t$:  
$$\Pi(s_t(v)) =s_t(\Pi(v))$$
for all $v\in \HH$ and $t\in \R$.

For $0<L<1$, the open double cone is defined as
$$\Cone_L = \{p\in \HH \mid d_\Kor(\zero, p) < L^{-1} |y(p)|\}.$$
This is a dilation-invariant set, and when $L$ is close to $1$, it is a small neighborhood of $\langle Y\rangle \setminus \{\zero\}$. An \emph{$L$--intrinsic Lipschitz graph} is an intrinsic graph $\Gamma_\phi$ such that $p\Cone_L\cap \Gamma_\phi=\emptyset$ for all $p\in \Gamma_\phi$. Equivalently, $\Gamma_\phi$ is $L$--intrinsic Lipschitz if and only if $\Lip(y|_{\Gamma_\phi})\le L$. We say that a function $\phi\from \HH \to \R$ is is an \emph{$L$--intrinsic Lipschitz function} if it is constant on cosets of $\paramY$ and $\Gamma_\phi$ is an $L$--intrinsic Lipschitz graph. 
(Again, some authors refer to $\Gamma_\phi$ as an \emph{entire intrinsic Lipschitz graph}. By Theorem~27 of \cite{NY1}, any subset of $\HH$ that satisfies the cone condition above can be extended to an entire intrinsic Lipschitz graph with the same intrinsic Lipschitz constant.)

If $A \subset V_0$ is a Borel set and $f\from A \to \R$ is a Borel function, $|A|$ will denote the Lebesgue measure of $A$ and $\int_A f(x) \ud x$ will denote the integral with respect to the Lebesgue measure. We will denote by $\mathcal{H}^3$ the $3$--dimensional Hausdorff measure on $\HH$ taken with respect to the Kor\'anyi metric. The following lemma is well known, see e.g. \cite{FS16} and \cite[Lemma 2.8]{CLY2}.
\begin{lemma} \label{lem:measure-equiv}
    Let $0 < L < 1$ and $\Gamma$ be a $L$--intrinsic Lipschitz graph and let $A \subset \Gamma$ be a measurable subset. Then $|\Pi(A)| \approx_L \cH^3(A).$
\end{lemma}

The following lemma will be used frequently in our proofs.
\begin{lemma}[{\cite[Lemma 2.3]{NY-Foliated}}] \label{lem:intrinsic-lipschitz-metric}
  Let \(0 < L < 1\) and let $\Gamma = \Gamma_\psi$ be an  $L$--intrinsic Lipschitz graph of a function $\psi\from U \to \R$.  Then for all  $v,w \in U$,
  \begin{align}
    |\psi(v) - \psi(w)| \leq \frac{2}{1-L} d_{\Kor}(\Psi_\psi(v),w\langle Y\rangle). \label{eq:lipschitz-metric}
  \end{align}
  In particular, for any $v\in U$ and any $s\in \R$ such that $v Z^s \in U$, 
  \begin{equation}\label{eq:lipschitz-vertical}
    |\psi(v) - \psi(v Z^s)| \le \frac{4}{1-L}\sqrt{|s|}.
  \end{equation}
\end{lemma}

Given a function $\psi \from \HH \to \R$ which is constant on cosets of $\paramY$ and a smooth function $f \from V_0 \to \R$, we define
\begin{equation}
\label{eq:horizder}
\partial_\psi f=\frac{\partial f}{\partial x}-\psi\frac{\partial f}{\partial z}.
\end{equation}
The differential operator $\partial_\psi$ defines a continuous vector field on $V_0$ whose $x$--coordinate is $1$, so the Peano existence theorem implies that there is at least one flow line of $\partial_\psi$ through every point of $V_0$. Such a flow line can be parametrized by $t\mapsto (t,0,g(t))$, where $g\from I \to \R$ (where $I$ is an interval) satisfies
\begin{align}
  g'(t) + \psi(t,0,g(t)) = 0,\quad \mbox{ for all }t\in I. \label{eq:cc-pde}
\end{align}
These flow lines are called  \textit{characteristic curves of $\Gamma_\psi$}, and we say that $g$ has a \emph{characteristic graph}.  
When $\psi$ is intrinsic Lipschitz, the characteristic curves of $\Gamma_\psi$ are exactly the $\Pi$--projections of horizontal curves $\gamma \from I \to \Gamma_\psi$ which satisfy $x(\gamma(t)) = t$ for all $t \in I$, see \cite[Lemma 2.6]{NY-Foliated}. Moreover, if $\psi$ is smooth then the characteristic curves of $\Gamma_\psi$ foliate $V_0$. However, this is not the case if $\psi$ is not smooth as characteristic curves could branch and rejoin, see e.g.\ \cite{BCSC15}.

The following lemma provides bounds on characteristic curves for intrinsic Lipschitz graphs.
\begin{lemma}[{\cite[Lemma 2.7]{NY-Foliated}}] \label{lem:intrinsic-lipschitz-cc}
Let $0 < L < 1$ and $\Gamma$ be an $L$--intrinsic Lipschitz graph. If $\gamma \from I \to V_0$  is a characteristic curve of \(\Gamma\) parameterized such that $x(\gamma(t )) = t$, then letting $g(t)=z(\gamma(t))$ we have:
  \begin{align}
    |g(t) - g(s) - g'(s) \cdot (t-s)| \leq \frac{L}{\sqrt{1-L^2}} \frac{(t-s)^2}{2}, \qquad \forall s,t \in I. \label{eq:g-linearize}
  \end{align}
\end{lemma}

\subsection{Foliated corona decompositions}
Foliated corona decompositions were recently introduced in \cite{NY-Foliated} to analyze the structure of intrinsic Lipschitz graphs in $\HH$. These decompositions partition an intrinsic Lipschitz graph $\Gamma_f$ into rectangular regions, called pseudoquads, on which $f$ is close to an affine function. In this section, we summarize the results from \cite{NY-Foliated} that we will need in this paper.

Let $\Gamma$ be an intrinsic Lipschitz graph.  A \emph{pseudoquad} $Q$ is a region of $V_0$ bounded by two vertical lines and two characteristic curves of $\Gamma$, i.e., a region of the form
$$\{(x,z)\in V_0\mid x\in I, g_1(x)\le z\le g_2(x)\},$$
where $I=[a,b]\subset \R$ is a closed interval and the $g_i$ are functions with characteristic graphs.  We say that $I$ is the \emph{base} of $Q$ and we call $g_1$ and $g_2$ the \emph{lower} and \emph{upper bounds} of $Q$.

Let $V$ be a vertical plane.  The horizontal curves of $V$ are parallel lines, and their projections to $V_0$ are parallel parabolas.  We call the pseudoquads of $V$ \emph{parabolic rectangles} for $V$; they are bounded by two parallel parabolas and two vertical lines.  If
$$R=\{(x,z)\in V_0\mid x\in I, h_1(x)\le z\le h_2(x)\}$$
is a parabolic rectangle, we define the \emph{width} of $R$ to be $\delta_x(R)=\ell(I)$ and the \emph{height} to be $\delta_z(R)=h_2-h_1$; since the graphs of $h_1$ and $h_2$ are parallel parabolas, these functions differ by a constant. The \emph{slope} of $R$, denoted $\slope(R)$, is $\slope(V)$; note that, by \eqref{eq:cc-pde} we have
\begin{equation}\label{eq:slope-c}
  h_i''(x) = - \slope(R)
\end{equation}
for all $i$ and $x$.

For $r>0$ and an interval $I=[a,b]$, let $rI$ be the scaling of $I$ around its center by a factor of $r$, i.e.,
$$rI=\left[\frac{a+b}{2}-\frac{r \ell(I)}{2},\frac{a+b}{2}+\frac{r \ell(I)}{2}\right].$$
For any $\rho>0$, let
\begin{align*}
  \rho R
  &=\left\{(x,z)\in V_0: x\in \rho I, z\in \rho^2 [h_1(x),h_2(x)]\right\}\\
  &=\left\{(x,z)\in V_0: x\in \rho I, \left|z-\frac{h_1(x)+h_2(x)}{2}\right| \le \frac{\rho^2 \delta_z(R)}{2}\right\};
\end{align*}
this is the concentric parabolic rectangle which has width $\delta_x(\rho R)=\rho \delta_x(R)$ and height $\delta_z(\rho R)=\rho^2 \delta_z(R)$.

For $0<\mu\le \frac{1}{32}$, a \emph{$\mu$--rectilinear pseudoquad} is a tuple $(Q,R)$, where $Q$ is a pseudoquad and $R$ is a parabolic rectangle with the same base $I$ such that, if $g_1$ and $g_2$ (resp.\ $h_1$ and $h_2$) are the lower and upper bounds of $Q$ (resp.\ $R$), then
\begin{equation}\label{eq:def rectilinear 1 and 2}
  \|g_i-h_i\|_{L_\infty(4I)}\le \mu \delta_z(R)
\end{equation}
for $i=1,2$. 
We often omit $R$ from the notation, referring to $(Q,R)$ as simply $Q$.  We define $\slope Q=\slope R$, $\delta_x(Q)=\delta_x(R)$, $\delta_z(Q)=\delta_z(R)$, and $\rho Q=\rho R$.  Note that $1Q = R$, so $1Q\ne Q$ in general. When $\mu=\frac{1}{32}$, we simply say that $Q$ is a rectilinear pseudoquad.

For any rectilinear pseudoquad $Q$, let $\alpha(Q)=\frac{\delta_x(Q)}{\sqrt{\delta_z(Q)}}$ be its \emph{aspect ratio}.  We use a square root here because the distance in the Heisenberg metric between the top and bottom of $Q$ is proportional to $\sqrt{\delta_z(Q)}$; this aspect ratio is scale-invariant.  

By the following lemma from \cite{NY-Foliated}, any pseudoquad that is sufficiently tall and skinny is rectilinear.
\begin{lemma}[{\cite[Lemma~5.3]{NY-Foliated}}]\label{lem:small-alpha-rectilinear}
  Let $\mu>0$, let $0<L<1$, and let \(\Gamma=\Gamma_f\) be an intrinsic $L$--Lipschitz graph.  There is an $A>0$ with the following property.  Let $Q$ be a pseudoquad for \(\Gamma\) and let $r=\delta_x(Q)$. Suppose that there is a point $v$ in the lower boundary of $Q$ and a point $vZ^s$ in the upper boundary such that $\frac{r}{\sqrt{s}} \le A$. Then there is a parabolic rectangle $R$ such that $(Q,R)$ is $\mu$--rectilinear.
\end{lemma}

A foliated corona decomposition of a graph $\Gamma$ is based on a collection of $\mu$--rectilinear foliated patchworks. These are similar to dyadic partitions or cubical patchworks in that they consist of a hierarchy of partitions of $\Gamma$ into pseudoquads, but an important difference is that different pieces can have different aspect ratios.

\begin{defn}\label{def:rectilinear foliated patchwork}
  Let $Q$ be a $\mu$--rectilinear pseudoquad. A \emph{\(\mu\)--rectilinear foliated patchwork} for $Q$ is a complete rooted binary tree \((\Delta,v_0)\) such that every $v$ in its vertex set \(\cV(\Delta)\) is associated to a $\mu$--rectilinear pseudoquad \((Q_v,R_v) \subset V_0\), $Q_{v_0}=Q$, and each vertex $v\in \cV(\Delta)$ is either \emph{vertically cut} or \emph{horizontally cut} in the following sense.

  Let $w$ and $w'$ be the children of $v$, let $I=[a,b]$ be the base of $Q_v$, and let $g_1$ and $g_2$ (resp.\ $h_1$ and $h_2$) be the lower and upper bounds of $Q_v$ (resp.\ $R_v$).  
  \begin{enumerate}
    \item 
      If $v$ is vertically cut, then $Q_w$ and $Q_{w'}$ are the left and right halves of $Q_v$, separated by the vertical line $x=\frac{a+b}{2}$.  That is,
      $$Q_{w}=\biggl\{(x,z)\in V_0\mid a\le x\le \frac{a+b}{2}, g_1(x)\le z\le g_2(x)\biggr\}$$
      $$Q_{w'}=\biggl\{(x,z)\in V_0\mid \frac{a+b}{2}\le x\le b, g_1(x)\le z\le g_2(x)\biggr\}.$$
      Similarly, $R_{w}=([a, \frac{a+b}{2}] \times \R) \cap R_v$, $R_{w'}=([\frac{a+b}{2}, b] \times \R) \cap R_v$.  Then $\delta_x(Q_w)=\delta_x(Q_{w'})=\frac{\delta_x(Q_v)}{2}$ and $\delta_z(Q_w)=\delta_z(Q_{w'})=\delta_z(Q_v)$.
    \item
      If $v$ is horizontally cut, then $Q_w$ and $Q_{w'}$ are the top and bottom halves of $Q_v$, separated by a characteristic curve.  That is, there is a function $c\from \R\to \R$ with characteristic graph, a quadratic function $k\from \R\to \R$, and a $d>0$ such that
      $$Q_{w}=\{(x,z)\in V_0\mid a\le x\le b, g_1(x)\le z\le c(x)\}$$
      $$Q_{w'}=\{(x,z)\in V_0\mid a\le x\le b, c(x)\le z\le g_2(x)\}$$
      $$R_{w}=\{(x,z)\in V_0\mid a\le x\le b, k(x)-d\le z\le k(x)\}$$
      $$R_{w'}=\{(x,z)\in V_0\mid a\le x\le b, k(x)\le z\le k(x)+d\}.$$
      Then $\delta_x(Q_w)=\delta_x(Q_{w'})=\delta_x(Q_v)$ and $\delta_z(Q_w)=\delta_z(Q_{w'})=d$.  Furthermore, by the $\mu$--rectilinearity of $Q_w$ and $Q_{w'}$,
      \begin{equation}\label{eq:horiz cut edge bounds}
        \max\left\{\|(k-d)-g_1\|_{L_\infty(4 I)}, \|k-c\|_{L_\infty(4 I)},\|(k+d)-g_2\|_{L_\infty(4 I)}\right\}\le \mu d.
      \end{equation}
  \end{enumerate}

  In either case, $Q_v=Q_{w}\cup Q_{w'}$ and the two halves $Q_w$ and $Q_{w'}$ have disjoint interiors.  Let \(\cVv(\Delta)\subset \cV(\Delta)\) be the set of vertically cut vertices and let \(\cVh(\Delta)\subset \cV(\Delta)\) be the set of horizontally cut vertices.
\end{defn}
As with pseudoquads, when $\mu$ is omitted, we take $\mu=\frac{1}{32}$, so any rectilinear foliated patchwork is at least $\frac{1}{32}$--rectilinear.

\begin{lemma} \label{lem:Q-2Q}
  Let $Q$ be a rectilinear pseudoquad.  Then $\frac{2}{3}Q \subseteq Q \subseteq 2Q$.
\end{lemma}

\begin{proof}
  The upper bound is Lemma 4.1 of \cite{NY-Foliated}.

  For the proof of the lower bound, let $R$ be the parabolic rectangle corresponding to $Q$, let $g_1,g_2\from \R \to \R$, $h_1,h_2\from \R \to \R$, and $k_1,k_2\from \R \to \R$ be functions whose graphs are the lower and upper bounds of $Q$, $R$, and $\frac{2}{3}Q$, respectively, and let $I$ be the base of $Q$. Then $h_i$ and $k_i$ are quadratics that all differ by additive constants and $|g_i(x) - h_i(x)| \leq \frac{1}{32} \delta_z(Q)$ for each $x \in I$ and $i \in \{1,2\}$.

  We have that $k_1 - h_1 = \frac{5}{18} \delta_z(Q) = h_2 - k_2$.  Thus, for each $x \in I$, we get
  \begin{align*}
    k_1(x) - g_1(x) \geq k_1(x) - h_1(x) - |g_1(x) - h_1(x)| > 0.
  \end{align*}
  Likewise, $k_2(x) < g_2(x)$ for all $x \in I$ and so $\frac{2}{3}Q \subseteq Q$.
\end{proof}

For any rooted tree $(T,v_0)$ and any $v\in \cV(T)$, let $\cC(v)=\cC^1(v)$ denote the set of \emph{children} of $v$ and let 
$$\cC^n(v)=\bigcup_{w\in \cC^{n-1}(v)} \cC(w)$$
be the set of $n$th generation descendants, letting $\cC^0(v)=\{v\}$.  Let
$$\cC^{\le n}(v)=\bigcup_{i=0}^{n} \cC^i(v)$$
and let $\cD(v)$ be the set of all descendants of $v$, including $v$ itself. We equip $\cV(T)$ with the usual partial order, so that $v\le w$ if $v$ is a descendant of $w$.

We record the following bounds; see \cite[Lemma 4.5]{NY-Foliated}
\begin{lemma}\label{lem:soft cut props}
  Let $\epsilon>0$.  There is a $\mu>0$ such that if \(Q\) is a $\mu$--rectilinear pseudoquad, then
  \begin{align}
      1-\epsilon \le \frac{\delta_x(Q) \cdot \delta_z(Q)}{|Q|}\le 1+\epsilon. \label{eq:soft cuts}
  \end{align}

  Further, if \(Q\) is horizontally or vertically cut as in Definition~\ref{def:rectilinear foliated patchwork} and \(Q'\) is a child of \(Q\), then
  $$\frac{1}{2}-\epsilon\le \frac{|Q'|}{|Q|} < \frac{1}{2}+\epsilon.$$
  If \(Q\) is horizontally cut, then $\delta_z(Q')=\frac{\delta_z(Q)}{2}$ and $\alpha(Q')= \sqrt{2}\alpha(Q)$.
  If \(Q\) is vertically cut, then $\delta_x(Q')=\frac{\delta_x(Q)}{2}$ and 
  $$\frac{1}{2}-\epsilon \le \frac{\alpha(Q')}{\alpha(Q)} \le \frac{1}{2}+\epsilon.$$
  In particular, when $Q$ is a rectilinear pseudoquad (i.e., $\mu=\frac{1}{32}$), then these inequalities hold for $\epsilon=\frac{1}{4}$.
\end{lemma}

We produce rectilinear foliated patchworks by repeatedly cutting a rectilinear pseudoquad into smaller pseudoquads. Any rectilinear pseudoquad $Q$ can be cut vertically into two $\mu$--rectilinear pseudoquads of the same height and half the width, but not every pseudoquad can be cut horizontally. Cutting along a characteristic curve through the center of $Q$ might not produce two rectilinear pseudoquads, since the upper and lower bounds of the new pseudoquads need not satisfy \eqref{eq:def rectilinear 1 and 2}.

The main technical result of \cite{NY-Foliated} shows that if a rectilinear pseudoquad is \emph{paramonotone}, then it can be cut horizontally into two rectilinear pseudoquads. Paramonotonicity is a condition based on the monotonicity introduced in \cite{CKMono}.  For any intrinsic Lipschitz graph $\Gamma$ and any $R>0$, there is a measure $\Omega^P_{\Gamma^+ ,R}$ on $V_0$ called the $R$--extended parametric normalized nonmonotonicity.  The full definition of $\Omega^P_{\Gamma^+ ,R}$ can be found in Section~8 of \cite{NY-Foliated}, but for $U\subset V_0$, $\Omega^P_{\Gamma^+ ,R}(U)$ measures the horizontal lines $L$ such that $L$ intersects $\Gamma$ multiple times in an $R$--neighborhood of $\Pi^{-1}(U)$.

This satisfies the kinematic formula
\begin{equation}\label{eq:kinematic}
  \sum_{i=-\infty}^\infty \Omega^P_{\Gamma^+,2^{-i}}(U)\lesssim_L |U|
\end{equation}
for any measurable $U\subset V_0$ \cite[Lem.~9.2]{NY-Foliated}. Furthermore, by inspection of the definition \cite[(155)]{NY-Foliated},
\begin{equation}\label{eq:Omega length change}
  \Omega^P_{\Gamma^+,R} \le \frac{R'}{R} \Omega^P_{\Gamma^+,R'}
\end{equation}
for all $R< R'$.

Let $Q$ be a pseudoquad of $\Gamma$.  We say that $\Gamma$ is \emph{$(\eta, R)$--paramonotone} on $rQ$ if the density of $\Omega^P_{\Gamma^+ ,R}$ on $rQ$ satisfies
\begin{equation}\label{eq:def-paramonotone}
  \frac{\Omega^P_{\Gamma^+,R\delta_x(Q)}(r Q)}{|Q|} \le \eta \alpha(Q)^{-4}.
\end{equation}
This condition is invariant under scalings, stretch maps, and shear maps. The results of \cite{NY-Foliated} show that paramonotone pseudoquads satisfy strong bounds on their characteristic curves.

\begin{prop}[{\cite[Prop.\ 7.2]{NY-Foliated}}]\label{prop:Omega-control}
  There is a universal constant $r>10$ such that for any $\lambda>0$ and $0<\zeta\le \frac{1}{32}$, there are $\eta, R > 0$ such that if \(\Gamma=\Gamma_f\) is the intrinsic Lipschitz graph of $f\from V_0\to \R$, and if $Q$ is a rectilinear pseudoquad for \(\Gamma\) such that \(\Gamma\) is $(\eta, R)$--paramonotone on $rQ$, then:
  \begin{enumerate}
  \item \label{it:omega control plane}
    There is a vertical plane $P\subset \H$ (a $\lambda$--approximating plane) and an affine function $F\from V_0\to \R$ such that $P$ is the intrinsic graph of $F$ and
    \begin{equation}
      \frac{\|F-f\|_{L_1(10 Q)}}{|Q|} \le \lambda \frac{\delta_z(Q)}{\delta_x(Q)}.
    \end{equation}
  \item \label{it:omega control characteristics}
    Let $u\in 4 Q$ and let $g_\Gamma, g_P\from \R\to \R$ be  such that $\{z=g_\Gamma(x)\}$ (respectively $\{z=g_P(x)\}$) is a characteristic curve for \(\Gamma\) (respectively $P$) that passes through $u$.  Then $$\|g_P-g_\Gamma\|_{L_\infty(4 I)}\le \zeta \delta_z(Q).$$
  \end{enumerate}
\end{prop}

We can thus produce a rectilinear foliated patchwork for a pseudoquad $Q$ by inductively cutting $Q$ into smaller and smaller pseudoquads. If a pseudoquad is paramonotone, we cut it horizontally. If not, we cut it vertically. We then repeat the process on the two new pseudoquads.

The resulting decomposition then satisfies certain bounds, which we describe below.
\begin{defn}\label{def:weighted Carleson}
  For $S\subset \cV(\Delta)$, let $W(S)=\sum_{w\in S} \alpha(Q_w)^{-4}|Q_w|$; we call this the \emph{weight} of \(S\).  We say that a rectilinear foliated patchwork \(\Delta\) satisfies a \emph{weighted Carleson packing condition} or that \(\Delta\) is \emph{$C$--weighted-Carleson} if every \(v\in \cV(\Delta)\) satisfies
  \begin{equation}\label{eq:def weighted Carleson}
    W(\{w\in \cVv(\Delta) \mid  w\le v\}) \le C |Q_v|.
  \end{equation}
\end{defn}

For any rectilinear pseudoquad $R$, $|R|\approx \delta_x(R)\delta_z(R)$, so
\begin{equation}\label{eq:weight deltax}
  W(\{R\}) = \alpha(R)^{-4} |R| \approx \left(\frac{\delta_x(R)}{\sqrt{\delta_z(R)}}\right)^{-4}|R|\approx \frac{\delta_z(R)^3}{\delta_x(R)^3}.
\end{equation}

\begin{defn}\label{def:lambda approximating}
  A set of \emph{$\lambda$--approximating planes} for $\Delta$ is a collection of vertical planes $P_v$ with corresponding vertical affine functions $l_v\from V_0\to \R$, one for each $v\in \cVh(\Delta)$, such that
  \begin{equation}\label{eq:def lambda approximating}
    |Q_v|^{-1}\|l_v-f\|_{L_1(10 Q_v)}
    \le \lambda \frac{\delta_z(Q_v)}{\delta_x(Q_v)}.
  \end{equation}
\end{defn}

\begin{defn} \label{def:foliated corona}
  Let $\mu_0>0$ and let $D\from \R^+\times \R^+\to \R^+$.  We say that an intrinsic Lipschitz graph \(\Gamma\) admits a \emph{$(D,\mu_0)$--foliated corona decomposition} if for every $0<\mu\le \mu_0$, every $\lambda>0$ and every $\mu$--rectilinear pseudoquad $Q_0\subset V_0$, there is a $\mu$--rectilinear foliated patchwork $\Delta$ for $Q_0$ such that $\Delta$ is $D(\mu,\lambda)$--weighted-Carleson and has a set of $\lambda$--approximating planes.  When $D$ and $\mu_0$ are not important, we simply say that $\Gamma$ admits a foliated corona decomposition.  
\end{defn}

\begin{thm}[{\cite[Thm.\ 7.5]{NY-Foliated}}]\label{thm:future-use}
  Any intrinsic Lipschitz graph admits a foliated corona decomposition.

  That is, let $r>10$ be as in Proposition~\ref{prop:Omega-control}. 
  For every $0<L<1$ and for every $0<\mu \le\frac{1}{32}$ and $\lambda >0$, there are $D=D(L,\mu,\lambda)$, $\eta=\eta(\mu,\lambda)$, and $R=R(\mu,\lambda)$, all positive, with the following properties. Suppose that \(\Gamma\subset \H\)  is an intrinsic $L$--Lipschitz graph and $Q\subset V_0$ is a $\mu$--rectilinear pseudoquad for $\Gamma$. Then there is a $\mu$--rectilinear foliated patchwork $\Delta$ for $Q$ such that $\Delta$ is $D$--weighted-Carleson and has a set of $\lambda$--approximating planes. Moreover, for all vertices $v\in \cV(\Delta)$, the associated pseudoquad $Q_v$ is horizontally cut if and only if \(\Gamma\) is $(\eta,R)$--paramonotone on $rQ$.

  Furthermore, $\eta$ and $R$ satisfy Proposition~\ref{prop:Omega-control} for $\zeta=\frac{1}{32r^2}$.
\end{thm}

(The condition that $\eta$ and $R$ satisfy Proposition~\ref{prop:Omega-control} is not part of the statement of Theorem~7.5 of \cite{NY-Foliated}, but in the proof, $\eta$ and $R$ are chosen to satisfy Proposition~\ref{prop:Omega-control} for a suitable $\zeta$ and $\lambda$.)

\section{Initial reductions}\label{sec:initial}
In this section, we reduce Theorem~\ref{thm:beta4} to a bound on a parametric $L_4$ version of $\beta_\Gamma$ on a pseudoquad for $\Gamma$. We first define $L_p$ $\beta$--numbers. For $E \subset \H, x \in E, r>0,$ and $p \in [1,\infty)$, let
\begin{equation}\label{eq:def betap} \beta_{p,E}(x,r):=\inf_{L \in \mathsf{VP}} \left[ r^{-3} \int_{B(x,r)\cap E} \left(\frac{d(y,L)}{r}\right)^p  \; \ud  \mathcal{H}^{3}(y)\right]^{1/p}.
\end{equation}
Recalling \eqref{eq:def beta}, we note that $\beta_{E}=\beta_{1,E}$. Using H\"older's inequality one easily sees that if $1\le p<q$ then
\begin{equation}
\label{eq:betapqrel}
\beta_{p,E}(x,r) \leq \left(\frac{\mathcal{H}^{3}(E \cap B(x,r))}{r^{3}} \right)^{1/p-1/q} \beta_{q,E}(x,r).
\end{equation}
In particular, if $E$ is Ahlfors $3$--regular, then $\beta_{p,E}(x,r)\lesssim \beta_{q,E}(x,r)$. Consequently, \eqref{eq:unbdbeta} holds when $\beta_{1,\Gamma}$ is replaced by $\beta_{p,\Gamma}$ for $p>1$. 

For any measurable function $\psi \from \HH \to \R$ that is constant on cosets of $\paramY$ and any $p  \geq 1$ we define a parametric version of $\beta$ by
\begin{equation}\label{eq:def-gamma}
  \gamma_{p,\psi}(v,r) = r^{\frac{-3-p}{p}} \inf_{h\in \Aff} \|\psi - h\|_{L_p(V(\Psi_\psi(v),r))}
\end{equation}
where
\begin{align*}
  V(v,r) := \Pi(B(v,r)),
\end{align*}
and  $\Aff$ denotes all functions $h: \HH \to \R$ of the form $h(v) = ax(v) + b$, for $a,b \in \R$. Observe that every vertical plane with finite slope is a graph of  a function in $\Aff$.   Note also that $\gamma_{p,\psi}(\cdot,r)$ is constant on cosets of $Y$.

The sets $V(v,r)$ are shaped like parallelograms in $V_0$, with slope depending on $y(v)$.
\begin{lemma} \label{lem:proj-quad}
  For any $r>0$ and $p\in \HH$, let $(x_0,0,z_0)=\Pi(p)$. Then
  \begin{equation}\label{eq:proj-quad}
    V(p,r) \subset \{(x,0,z) \mid |x-x_0|\le r, |z - z_0 + y(p) (x-x_0)| \le r^2\}.
  \end{equation}
\end{lemma}
\begin{proof}
  We first consider the case $p=\mathbf{0}$. If $q\in B(\mathbf{0},r)$, then $\Pi(q) = qY^{-y(q)} \in B(\mathbf{0},2r)\cap V_0$. In particular, $|z(\Pi(q))|\le r^2$. Since $x(\Pi(q))=x(q)\in [-r,r]$, this implies $V(\mathbf{0},r)\subset R_r$, where $R_r = [-r,r]\times 0 \times [-r^2,r^2]$.
  
  Now consider an arbitrary $p\in \HH$. Then $\Pi(pq) = pqY^{-y(p)-y(q)} = \Pi(p\Pi(q))$ for any $q\in \HH$, so
  $$V(p,r) = \Pi(B(p,r)) = \Pi(p B(\mathbf{0},r)) = \Pi(p V(\mathbf{0},r)) \subset \Pi(p R_r).$$
  We write $p=(x_0,0,z_0)Y^{y(p)}$; then
  $$V(p,r) \subset \Pi(p R_r) = (x_0,0,z_0)Y^{y(p)} R_r Y^{-y(p)}.$$
  By \eqref{eq:heis-mult}, $Y^{y(p)} (x,0,z) Y^{-y(p)} = (x, 0, z - x y(p))$, so this implies \eqref{eq:proj-quad}.
\end{proof}

When $\psi$ is intrinsic Lipschitz, $V(x,r)$ and $\Pi(B(x,r) \cap \Gamma_\psi)$ are comparable, and $\beta_{\Gamma_\psi}$ and $\gamma_\psi$ are comparable.
\begin{lemma}[{\cite[Lemma 2.7]{CLY2}}] \label{lem:V-comparable}
  Let $\psi$ be a $L$-intrinsic Lipschitz function, let $p \in \Gamma_\psi$, and let $r > 0$.  There is a $c > 1$ depending on $L$ such that
  \begin{align}
    V(p,c^{-1}r) \subseteq \Pi(B(p,r) \cap \Gamma_\psi) \subset V(p,r). \label{eq:V-comparable}
  \end{align}
  In particular, if $x,y\in V_0$ and $y\in V(\Psi_\psi(x),r)$, then $x\in V(\Psi_\psi(y),cr)$.
\end{lemma}

\begin{lemma}\label{lem:parametric-beta}
  Let $0 < L < 1$ and $p \geq 1$. There is a $C:=C(L)>1$ such that for any $L$--intrinsic Lipschitz function  $\psi\from \HH\to \R$, any $x\in \HH$, and any $r>0$,
  \begin{align}
    \gamma_{p,\psi}(x,r/C) \lesssim_{L,p} \beta_{p,\Gamma_\psi}(\Psi_\psi(x),r) \lesssim_{L,p} \gamma_{p,\psi}(x,Cr). \label{eq:parametric-beta}
  \end{align}
\end{lemma}
For $p=1$, the previous lemma was essentially proved in \cite[Lemma 4.2]{CLY2}. For $p>1$ the proof is very similar and we omit it.

Then Theorem \ref{thm:beta4} is a special case of the following theorem.
\begin{thm}\label{thm:betaprange}
 Let \(\Gamma\) be an intrinsic \(L\)--Lipschitz graph in $\H$ and let $p \in [1,4]$.  Then for any $y\in \Gamma$ and any $R>0$,
  \begin{equation}
  \label{eq:betaprange}
    \int_{B(y,R) \cap \Gamma} \int_0^R \beta_{p,\Gamma}(x,r)^4\;\frac{\mathrm{d} r}{r}\ud \cH^3(x) \lesssim_{L} R^3.
  \end{equation}
\end{thm}
We have not considered if the range $[1,4]$ is sharp, and most likely it is not. By Dorronsoro's work \cite{dor} the strong geometric lemma for $m$--dimensional Lipschitz graphs (or Lipschitz functions on $\R^m$) holds for $\beta_p$  with $p<\infty$ in $\R^2$ and for $p<\frac{2m}{m-2}$ in $\R^n, n>2$. Unpublished examples of Fang and Jones show that these ranges are sharp. Recently, F\"assler and Orponen \cite{FO} extended Dorronsoro's techniques to Lipschitz functions from $\H_n$ to $\R$. Their approach suggests that \eqref{eq:betaprange} might hold for $1\le p < 6$, but we will not pursue this here.\footnote{In \cite{CLY1}, we proved Theorem~\ref{thm:beta2 highdim} for $\beta_{2,\Gamma}$, so \eqref{eq:betahighdim} holds when $\beta_{1,\Gamma}$ is replaced by $\beta_{p,\Gamma}$ with $p \in [1,2]$. Likewise, this range is not sharp and for the same reasons it is likely that \eqref{eq:betahighdim} holds for $\beta_{p,\Gamma}$ with $1\le p < \frac{2(2n+1)}{(2n+1)-2} = \frac{4n+2}{2n-1}$.}

In the rest of this section, we will show that Theorem~\ref{thm:betaprange} is a consequence of the following bound.
\begin{prop} \label{prop:Q-carleson}
  Let $0<L<1$. There are $\tau > 0$ and $c>0$ such that if \(\Gamma = \Gamma_f\) is an $L$--intrinsic Lipschitz graph and $Q \subseteq V_0$ be a rectilinear pseudoquad with $\alpha(Q) \le c$, then 
  \begin{align}
    \int_{\frac{1}{3}Q} \int_0^{\tau \delta_x(Q)} \gamma_{4,f}(x,r)^4 \frac{\ud r}{r} \ud x \lesssim_{L} |Q|. \label{eq:Q-carleson}
  \end{align}
\end{prop}

To show that Proposition~\ref{prop:Q-carleson} implies Theorem~\ref{thm:betaprange}, we will need the following lemma.
\begin{lemma} \label{lem:SAR-conseq}
   Let \(0 < \kappa < 1\) and \(0<L<1\). There is a $c>1$ depending on $\kappa$ and $L$ such that for any $L$--intrinsic Lipschitz graph $\Gamma=\Gamma_{\psi}$, any $x\in \Gamma$, and any $r>0$, there is a rectilinear pseudoquad $Q$ such that $V(x,r)\subset \kappa Q$, $\delta_x(Q) = 2\kappa^{-1} r$, and $\delta_z(Q) \le c r^2$.
\end{lemma}
\begin{proof}
  After a left-translation, we may suppose that $x=\mathbf{0}$. Let 
  $a = a(L) \in (0,1)$ be a small number to be chosen later and let $s=2 \kappa^{-2} a^{-2} r^2$. 
  Let $g_1,g_2\from \R\to \R$ be functions with characteristic graphs such that $g_1(0)=-s$ and $g_2(0)=s$, let 
  $$Q=\{(x,0,z) : |x|\le \kappa^{-1} r \text{ and } z\in [g_1(x),g_2(x)]\},$$
  and let $R=[-\kappa^{-1} r, \kappa^{-1} r] \times [-s,s]$ so that $\alpha(R) = \frac{2\kappa^{-1} r}{\sqrt{2s}} = a$. We claim that $(Q,R)$ is rectilinear and that $V(x,r)\subset Q$.
  
  By \eqref{eq:cc-pde} and Lemma~\ref{lem:intrinsic-lipschitz-metric}, 
  $$|g_i'(0)| = |\psi(0,0,g_i(0))| \le \frac{4}{1-L} \sqrt{s}\le \frac{8r\kappa^{-1} a^{-1}}{1-L}.$$
  By Lemma~\ref{lem:intrinsic-lipschitz-cc}, for $t\in [-4\kappa^{-1} r, 4\kappa^{-1} r]$, 
  $$\frac{|g_i(t)-g_i(0)|}{s} \le \frac{1}{2\kappa^{-2}a^{-2}r^2}\left(|g_i'(0)| 4\kappa^{-1} r + \frac{16L\kappa^{-2}}{\sqrt{1-L^2}} r^2\right)
     \le \frac{16a}{1-L} + \frac{8L a^2}{\sqrt{1-L^2}}.$$
  We choose $a(L)\in (0,1)$ small enough that
  $$|g_i(t)-g_i(0)| \le \frac{s}{16} = \frac{1}{32} \delta_z(R),$$
  so $(Q,R)$ is rectilinear.
  By Lemma~\ref{lem:proj-quad}, $V(x,r) \subset [-r,r]\times [-r^2,r^2]$, so
  $$V(x,r) \subset [-r,r]\times [-2a^{-2}r^2, -2a^{-2}r^2] = \kappa Q,$$
  and $\delta_z(Q) = 4 \kappa^{-2} a^{-2} r^2\lesssim_{\kappa,L} r^2$.
\end{proof}

This lets us prove Theorem~\ref{thm:betaprange}, assuming Proposition~\ref{prop:Q-carleson}.
\begin{proof}[Proof of Theorem \ref{thm:betaprange}]
  Let $\Gamma = \Gamma_f$ be a $L$-intrinsic Lipschitz graph, $y \in \Gamma$, and $R > 0$. We first prove that
  \begin{align}
    \int_{V(y,R)} \int_0^R \gamma_{4,f}(x,r)^4 \frac{\ud r}{r} \ud x \lesssim_{L} R^3. \label{eq:gamma-carleson}
  \end{align}

  Let $c, \tau > 0$ be the constants from Proposition~\ref{prop:Q-carleson}. By Lemma~\ref{lem:SAR-conseq}, there is a rectilinear pseudoquad $Q$ with $\delta_x(Q) \lesssim_L R$ and $\delta_z(Q) \lesssim_L R^2$ such that $V(y,R) \subset \frac{1}{3}Q$ and $\tau\delta_x(Q) \geq R$. Then
  \begin{align*}
    \int_{V(y,R)} \int_0^R \gamma_{4,f}(x,r)^4 \frac{\ud r}{r} \ud x & \le \int_{\frac{1}{3} Q} \int_0^{\tau \delta_x(Q)} \gamma_{4,f}(x,r)^4 \frac{\ud r}{r} \ud x \\
    &\overset{\eqref{eq:Q-carleson}}{\lesssim_{L}} |Q| \approx_L R^3.
  \end{align*}

  Let $C = C(L) > 1$ be the constant from Lemma \ref{lem:parametric-beta}. Lemma \ref{lem:measure-equiv} tells us that $(\Psi_f)_*(|\cdot|) \approx_L \cH^3$ on $\Gamma$, so 
  \begin{align*}
    \int_{B(y,R) \cap \Gamma} &\int_0^R \beta_{4,\Gamma}(x,r)^4 \frac{\ud r}{r} \ud \cH^3(x) \\
    & \lesssim_L \int_{\Pi(B(y,R) \cap \Gamma)} \int_0^R \beta_{4,\Gamma}(\Psi_f(x),r)^4 \frac{\ud r}{r} \ud x \\
    &\overset{\eqref{eq:parametric-beta}}{\lesssim_L} \int_{V(y,R)} \int_0^R \gamma_{4,f}(x,Cr)^4 \frac{\ud r}{r} \ud x \\
    &\leq \int_{V(y,CR)} \int_0^{CR} \gamma_{4,f}(x,r)^4 \frac{\ud r}{r} \ud x \\
    &\overset{\eqref{eq:gamma-carleson}}{\lesssim_L} R^3.
  \end{align*}
  By \eqref{eq:betapqrel}, this implies the theorem for $p\in [1,4]$.

\end{proof}

\section{Foliated coronizations}\label{sec:coronizations}

In this section, we will use foliated corona decompositions to define \emph{foliated coronizations}, which are rectilinear foliated patchworks with some improved properties that make them easier to use for the arguments in this paper. 

\begin{defn}\label{def:coronization}
  Let $\lambda>0$, let $r>10$ be the universal constant in Proposition~\ref{prop:Omega-control}, and let $0<\mu\le\frac{1}{32r^2}$. If $\Delta$ is a $\mu$--rectilinear foliated patchwork with horizontally cut root that is $D_0$--weighted-Carleson and has a set of $\lambda$--approximating planes, we call $\Delta$ a \emph{foliated coronization}.  When the constants are important, we will write that $\Delta$ is a $(D_0, \mu, \lambda)$--foliated coronization.
  Furthermore, if there are $\eta, R>0$ that satisfy Proposition~\ref{prop:Omega-control} for $\zeta=\frac{1}{32r^2}$ and such that for all $v\in \cV(\Delta)$, $Q_v$ is horizontally cut if and only if $\Gamma$ is $(\eta,R)$--paramonotone on $rQ_v$, then we say that $\Delta$ is \emph{$(\eta,R)$--paramonotone stopped}. We will abbreviate paramonotone stopped foliated coronization as PSFC.
\end{defn}

We can construct such coronizations using the following lemma.
\begin{lemma}\label{lem:low aspect pseudoquads}
  Let $0<L<1$ and let $\eta,r,R,\mu>0$. There is an $\alpha_{\mathrm{min}}\in(0,1)$ such that for any $L$--intrinsic Lipschitz graph $\Gamma$, if $Q$ is a rectilinear pseudoquad for $\Gamma$ and $\alpha(Q)\le \alpha_{\mathrm{min}}$, then $Q$ is $\mu$--rectilinear and $\Gamma$ is $(\eta,R)$--paramonotone on $rQ$.
\end{lemma}
\begin{proof}
  Let $A$ be as in Lemma~\ref{lem:small-alpha-rectilinear}, so that if $\alpha(Q)\le \frac{A}{2}$, then $Q$ is $\mu$--rectilinear.
  
  Let $i>0$ be such that $2^{i-1}\le R\delta_x(Q)<2^i$.  By \eqref{eq:kinematic} and \eqref{eq:Omega length change}, there is a $b>0$ such that
  $$\Omega^P_{\Gamma^+,R\delta_x(Q)}(rQ) \le 2 \Omega^P_{\Gamma^+,2^i}(rQ) \le b |rQ|.$$
  Let
  $$\alpha_{\mathrm{min}}=\min\{\frac{A}{2}, (4 b r^3\eta^{-1})^{-\frac{1}{4}}\}.$$

  By Lemma~\ref{lem:soft cut props}, $|rQ|\le 2 r^3 |Q|$, so if $\alpha(Q)>\alpha_{\mathrm{min}}$, then 
  $$\Omega^P_{\Gamma^+,R\delta_x(Q)}(rQ) \le 2 b r^3 |Q| \le \eta \alpha(Q)^{-4} |Q|,$$
  and $\Gamma$ is $(\eta,R)$--paramonotone on $rQ$.
\end{proof}

Combining Lemma~\ref{lem:low aspect pseudoquads} and Theorem~\ref{thm:future-use} yields the following.
\begin{lemma}\label{lem:coronizations-exist}
  Let $0<L<1$. There is an $\alpha_{\mathrm{min}}\in(0,1)$ such that for any intrinsic $L$--Lipschitz graph $\Gamma$, and any rectilinear pseudoquad $Q\subset V_0$ with $\alpha(Q) \le \alpha_{\mathrm{min}}$, there is a PSFC with root $Q$.

  Specifically, let $r$ be as in Proposition~\ref{prop:Omega-control}.
  Let $\lambda >0$, and let $0<\mu\le\frac{1}{32r^2}$. Let $D=D(L,\mu,\lambda)$, $\eta=\eta(\mu,\lambda)$, and $R=R(\mu,\lambda)$ be as in Theorem~\ref{thm:future-use}. Let $\alpha_{\mathrm{min}}=\alpha_{\mathrm{min}}(\eta, \mu, r, R, L)$ be as in Lemma~\ref{lem:low aspect pseudoquads}. Then $\eta$ and $R$ satisfy Proposition~\ref{prop:Omega-control} for $\zeta=\frac{1}{32r^2}$. Furthermore, for any intrinsic $L$--Lipschitz graph $\Gamma$, and any rectilinear pseudoquad $Q\subset V_0$ with $\alpha(Q) \le \alpha_{\mathrm{min}}$, there is a $(D,\mu, \lambda)$--foliated coronization $\Delta$ of $Q$ which is $(\eta,R,r)$--paramonotone stopped.
\end{lemma}
\begin{proof}
  Let $Q\subset V_0$ be as above. By Lemma~\ref{lem:low aspect pseudoquads}, $\Gamma$ is $(\eta,R)$--paramonotone on $rQ$, and by Theorem~\ref{thm:future-use}, there is a $\mu$--rectilinear foliated patchwork $\Delta$ for $Q$ that is $D$--weighted-Carleson, has a set of $\lambda$--approximating planes, and is $(\eta,R,r)$--paramonotone stopped. In particular, $Q$ is horizontally cut, so $\Delta$ is a PSFC as desired.
\end{proof}

These decompositions satisfy several nice properties. First, the pseudoquads of a PSFC have aspect ratios that are bounded below.
\begin{lemma}\label{lem:aspect-bounds}
  Let $0<L<1$, let $\Gamma$ be an intrinsic $L$--Lipschitz graph $\Gamma$, and let $Q\subset V_0$ be a pseudoquad for $\Gamma$. Let $\Delta$ be a PSFC for $Q$ and let $\alpha_{\mathrm{min}}$ be as in Lemma~\ref{lem:low aspect pseudoquads}. Then:
  \begin{itemize}
  \item for all $v\in \cV(\Delta)$, $\alpha(Q_v) \ge \min\{\alpha(Q), \frac{\alpha_{\mathrm{min}}}{4}\}$,
  \item there are only finitely many $v\in \cV(\Delta)$ such that $\alpha(Q_v)<\frac{\alpha_{\mathrm{min}}}{4}$, and
  \item if $\delta_z(Q_v) \le \frac{16 \delta_x(Q)^2}{\alpha_{\mathrm{min}}^2}$, then $\alpha(Q_v) \ge \frac{\alpha_{\mathrm{min}}}{4}$.
  \end{itemize}
\end{lemma}
\begin{proof}
  Let $v_0$ be the root of $\Delta$.
  
  Suppose that $v\ne v_0$ and $\alpha(Q_v) < \frac{\alpha_{\mathrm{min}}}{4}$. Let $p$ be the parent of $v$.  By Lemma~\ref{lem:soft cut props}, $\alpha(Q_v)<\alpha_{\mathrm{min}}$. Therefore, by Lemma~\ref{lem:low aspect pseudoquads}, $\Gamma$ is $(\eta,R)$--paramonotone on $rQ_p$ and $Q_p$ is horizontally cut. Therefore, $\alpha(Q_p) < \alpha(Q_v) < \frac{\alpha_{\mathrm{min}}}{4}$. The same argument thus holds for $Q_p$. By induction, every ancestor $a$ of $v$ is horizontally cut and $\alpha(Q_a) < \frac{\alpha_{\mathrm{min}}}{4}$.
  Consequently, $\alpha(Q_v)\ge \alpha(Q)$. This proves the first part of the lemma.
  
  Furthermore, if $\alpha(Q_v) < \frac{\alpha_{\mathrm{min}}}{4}$, then $\delta_x(Q_v) = \delta_x(Q)$ and thus
  $$\delta_z(Q_v) = \frac{\delta_x(Q_v)^2}{\alpha(Q_v)^2} > \frac{16 \delta_x(Q)^2}{\alpha_{\mathrm{min}}^2}.$$
  There are only finitely many such pseudoquads in $\Delta$, so this proves the rest of the lemma.
\end{proof}

Second, expansions of the pseudoquads of $\Delta$ are nested.
\begin{lemma}[{\cite[Lemma~4.7]{NY-Foliated}}] \label{lem:child-hierarchy}
  Let $\Delta$ be a PSFC and suppose $v,w\in \cV(\Delta)$ satisfy $w\le v$. Let $0<r'\le r$. Then $r'Q_w\subset r'Q_v$.
\end{lemma}
(This result is proved for sufficiently small $\mu$ in \cite{NY-Foliated}, but inspection of the proof shows that $\mu\le \frac{1}{32r^2}$ is enough.)

Third, every pseudoquad has an approximating plane that satisfies Proposition~\ref{prop:Omega-control}.
\begin{lemma} \label{lem:full approximating planes}
  Let $\Delta$ be a PSFC as above.  For $w\in \cV(\Delta)$, let $m$ be the minimal horizontally-cut ancestor of $w$, where $m=w$ if $w\in \cVh(\Delta)$. Since the root of $\Delta$ is horizontally-cut, such an ancestor exists. Let $l_w:=l_m$ and $P_m:=P_m$, where $l_m$ and $P_m$ are as in Definition~\ref{def:lambda approximating}, and let $I$ be the base of $Q_w$. Then $Q_w$ and $P_w$ satisfy Proposition~\ref{prop:Omega-control}, i.e.,
  \begin{equation}\label{eq:L1-bound}
    \|l_w-f\|_{L_1(10 Q_w)}\lesssim \lambda \frac{\delta_z(Q_w)}{\delta_x(Q_w)} |Q_w|,
  \end{equation}
  and for any $u\in 4 Q$, if $g_\Gamma, g_{P_w}\from \R\to \R$ are functions such that $\{z=g_\Gamma(x)\}$ (respectively $\{z=g_{P_w}(x)\}$) are characteristic curves for \(\Gamma\) (respectively $P_w$) that pass through $u$, then 
  \begin{equation}\label{eq:approx-ccurve}
    \|g_{P_w}-g_\Gamma\|_{L_\infty(4 I)}\le \frac{1}{32 r^2} \delta_z(Q).
  \end{equation}
\end{lemma}
\begin{proof}
  For $w\in \cVh(\Delta)$, the lemma follows from Proposition~\ref{prop:Omega-control}. We thus consider $w\in \cVv$. 
  Since $Q_m$ is rectilinear, $\delta_z(Q_w)\approx \delta_z(Q_m)$. By \eqref{eq:def lambda approximating},
  \begin{multline*}
    \|l_w-f\|_{L_1(10Q_w)} 
    \le \|l_m-f\|_{L_1(10Q_m)} 
    \le \lambda \frac{\delta_z(Q_m)}{\delta_x(Q_m)} |Q_m| \\
    \approx \lambda \delta_z(Q_m)^2
    \approx \lambda \delta_z(Q_w)^2
    \approx \lambda \frac{\delta_z(Q_w)}{\delta_x(Q_w)} |Q_w|.
  \end{multline*}
  
  Suppose that $u\in 4Q_w$ and that $\{z=g_\Gamma(x)\}$ and $\{z=g_{P_w}(x)\}$ are characteristic curves through $u$. By Lemma~\ref{lem:child-hierarchy}, $u\in 4Q_m$, so \eqref{eq:approx-ccurve} follows from Proposition~\ref{prop:Omega-control} applied to $Q_m$.
\end{proof}

\section{$L_p$ approximation by piecewise affine functions}\label{sec:lp-approx}

Now we prove $L_p$ bounds on the pseudoquads of a foliated coronization for $\Gamma_f$. By Lemma~\ref{lem:full approximating planes}, these pseudoquads have $\lambda$--approximating planes, but these planes only approximate $f$ in $L_1$. In this section, we will a family of  piecewise affine functions $g_S$ that approximate $f$ and we will bound $\|f-g_S\|_p$ for $1\le p < 5$.

We will prove the following proposition.
\begin{prop}\label{prop:Lp approximating planes}
  Let $\Gamma=\Gamma_f$ be an $L$--intrinsic Lipschitz graph. Let $Q$ be a pseudoquad of \(\Gamma\) and let $\Delta$ be a PSFC with root $Q$. Let $l_v,v\in \cV(\Delta)$ be the approximating planes for $\Delta$.
  There is a $c>0$ depending on $L$ and the parameters of $\Delta$ such that for every $v\in \cVh(\Delta)$ and $1\le p <5$,
  \begin{equation}\label{eq:Lp lambda approximating}
    \|l_v-f\|_{L_p(Q_v)} \le \frac{c}{5-p} \lambda \frac{\delta_z(Q_v)}{\delta_x(Q_v)} |Q_v|^{\frac{1}{p}}.
  \end{equation}
\end{prop}
We prove Proposition~\ref{prop:Lp approximating planes} by approximating $f$ by piecewise-affine functions. 

We will need a few definitions. Let $\Delta$ be a rooted tree.
A \emph{coherent} set $S\subset \cV(\Delta)$ is a subset with the following properties:
\begin{enumerate}
\item $S$ has a unique maximal element $M=\mathsf{max}(S)\in S$.
\item If $v\in S$ and $w\in \cV(\Delta)$ satisfies $v<w<M$, then $w\in S$.
\item \label{it:coherence siblings}
  If $v\in S$, then either all of the children of $v$ are contained in $S$ or none of them are.
\end{enumerate}
A \emph{partition} of a subset $U\subset V_0$ into pseudoquads is a finite collection $Q_1,\dots, Q_k$ of pseudoquads such that $\bigcup Q_i=U$ and such that the interiors of the $Q_i$'s are pairwise disjoint.  A coherent subset of a foliated patchwork corresponds to a partition of a set into pseudoquads. The following lemma is Lemma~6.3 of \cite{NY-Foliated}.
\begin{lemma}\label{lem:coherent partitions}
  Let $\Delta$ be a rectilinear foliated patchwork and let $S\subset \cV(\Delta)$ be coherent.  Let $M=\max S$ be the maximal element of \(S\) and let $\min S$ be the set of minimal elements of $S$.  Let
  $$F_1=F_1(S)=\{p\in Q_M\mid \textup{there are infinitely many $v\in S$ such that $p\in Q_v$}\}$$
  and let $F_2=F_2(S)=Q_M\setminus F_1$.  Then
  \begin{equation}\label{eq:coherent partitions}
    Q_M=F_1\cup \bigcup_{w\in \min S} Q_w,
  \end{equation}
  and the interiors of the $Q_w$'s are pairwise disjoint and disjoint from $F_1$.  If $S$ is finite, then $\min S$ is a partition of $Q_M$.
\end{lemma}

We use these partitions to define approximations of $f$.

\begin{defn}[piecewise-affine approximations]\label{def:gS}
Let $\Delta$ be a paramonotone stopped foliated coronization (a PSFC) for $\Gamma_f$. By Lemma~\ref{lem:full approximating planes}, there is an associated collection of affine functions $l_w,w\in \cV(\Delta)$. For any coherent set $S\subset \cV(\Delta)$, we define a function $g_S\in L_\infty(Q_{\mathsf{max}(S)})$ as follows.  Let $F_1,F_2$ be as in Lemma~\ref{lem:coherent partitions} and define $g_S$ so that $g_S|_{F_1}=f|_{F_1}$ and $g_S|_{Q_w}=l_w|_{Q_w}$ for all $w\in \mathsf{min}(S)$; this is well-defined away from the boundaries of the $Q_w$'s.
\end{defn}

We claim that as $T$ gets larger and larger, $g_T$ converges to $f$ in $L_p$. We first prove convergence in $L_1$.
\begin{lemma}\label{lem:coherent convergence}
  Let $\Delta$ be a PSFC for a pseudoquad $Q$.  Let $v\in \cV(\Delta)$ and let $S\subset \cV(\Delta)$ be a coherent subset with $\max S=v$.  Let $T_1,T_2,\dots\subset S$ be a sequence of coherent sets with $\max T_i=v$ such that $T_1\subset T_2\subset \dots$ and $\bigcup T_i=S$.  Then
  $$\lim_{i\to \infty} \|g_{T_i}-g_S\|_{L_1(Q_v)}=0.$$
\end{lemma}

\begin{proof}
  Let $T$ be a coherent set with $u=\max T$. We first show that
  \begin{equation}\label{eq:four-thirds}
    \|g_T - f\|_{L_1(Q_u)} \lesssim |Q_u|^{\frac{4}{3}},
  \end{equation}
  with implicit constant depending on the parameters of $\Delta$.
  Let $C=\min\left\{\frac{\alpha_{\mathrm{min}}}{4}, \alpha(Q_{v_0})\right\}$ so that $\alpha(Q_w)\ge C$ for all $w\in \cV(\Delta)$ by Lemma~\ref{lem:aspect-bounds}. Since $Q_w$ is rectilinear, we have $|Q_w|\approx \delta_z(Q_w)\delta_x(Q_w)$, so 
  $$\delta_z(Q_w)\approx \alpha(Q_w)^{-\frac{2}{3}} |Q_w|^{\frac{2}{3}} \le C^{-\frac{2}{3}} |Q_w|^{\frac{2}{3}}$$
  and
  $$\delta_x(Q_w)\approx \alpha(Q_w)^{\frac{2}{3}} |Q_w|^{\frac{1}{3}} \ge C^{\frac{2}{3}} |Q_w|^{\frac{1}{3}}.$$

  Let $l_w$ be the $\lambda$--approximating planes for $\Delta$.  Then, by Lemma~\ref{lem:full approximating planes}, 
  $$\|l_w - f\|_{L_1(Q_w)} \lesssim \lambda \frac{\delta_z(Q_w)}{\delta_x(Q_w)} |Q_w| \lesssim \lambda C^{-\frac{4}{3}} |Q_w|^{\frac{4}{3}}.$$
  By Lemma~\ref{lem:coherent partitions},
  $$\|g_{T}-f\|_{L_1(Q_u)} = \sum_{m\in \min T} \|g_{T}-f\|_{L_1(Q_m)} \lesssim \sum_{m\in \min T}  \lambda C^{-\frac{4}{3}} |Q_u|^{\frac{1}{3}} |Q_m| \le \lambda C^{-\frac{4}{3}} |Q_u|^{\frac{1}{3}}|Q_u|,$$
  which implies \eqref{eq:four-thirds}.

  Now we consider $g_S$ and $g_{T_i}$. Let $F_1(S)$ and $F_1(T_i)$ be as in Lemma~\ref{lem:coherent partitions} for $S$ and $T_i$; note that $F_1(T)\subset F_1(S)$, so $g_{T_i}=g_S=f$ on $F_1(T)$. Let $M_i = \min T_i$. If $u \in M_i \cap \min S$, then $g_S = g_{T_i} = l_u$ on $\inter Q_u$. Let $N_i = M_i\setminus \min S$. Then for each $u\in N_i$, the intersections $T_i\cap \cD(u)$ and $S\cap \cD(u)$ are coherent sets containing $u$, and 
  $$\|g_{T_i}-g_S\|_{L_1(Q_v)} = \sum_{u\in N_i} \|g_{T_i}-g_S\|_{L_1(Q_u)} = \sum_{u\in N_i} \|g_{T_i\cap \cD(u)}-g_{S\cap \cD(u)}\|_{L_1(Q_u)}.$$
  Let $a_i= \max_{u\in N_i} |Q_u|$. By \eqref{eq:four-thirds},
  $$\|g_{T_i}-g_S\|_{L_1(Q_v)} \lesssim a_i^{\frac{1}{3}} \sum_{u\in N_i} |Q_u|\le a_i^{\frac{1}{3}} |Q_v|,$$
  so it suffices to show that $\lim_i a_i=0$.

  Let $\epsilon>0$. By Lemma~\ref{lem:soft cut props}, for any $\epsilon>0$, there are only finitely many $w\in S$ such that $|Q_w|>\epsilon$. Suppose that $i$ is large enough that $T_i$ contains every such $w$. Then any $u\in N_i$ has a child $u'\in S$ with $|Q_{u'}|\le \epsilon$. By Lemma~\ref{lem:soft cut props}, $|Q_u|\le 4\epsilon$, so $a_i\le 4\epsilon$. Letting $\epsilon$ go to zero, we find that $\lim_i a_i=0$ and thus $\lim_i \|g_{T_i}-g_S\|_{L_1(Q_v)}=0$.
\end{proof}

We thus consider how $g_S$ changes when we enlarge $S$. First, we consider adding the children of a vertex $w\in \mathsf{min}(S)$ to $S$. This corresponds to cutting a pseudoquad in the corresponding partition into two pieces.
\begin{lemma}\label{lem:child differences}
  Let $S$ be a coherent set and let $w\in \mathsf{min}(S)$.  Then $S'=S\cup \cC(w)$ is coherent, $\supp(g_S-g_{S'})\subset Q_w$, and 
  $$\|g_S-g_{S'}\|_{L_\infty(Q_w)}\lesssim \lambda \frac{\delta_z(Q_u)}{\delta_x(Q_u)}.$$
\end{lemma}
\begin{proof}
  It follows from the definitions that $g_S$ and $g_{S'}$ agree outside of $Q_w$.  Let $u$ be a child of $w$.  Then $\delta_x(Q_u)\approx \delta_x(Q_w)$, $\delta_z(Q_u)\approx \delta_z(Q_u)$, and $|Q_u|\approx |Q_w|$, so 
  $$\|l_w-l_u\|_{L_1(10 Q_u)} \le \|l_w-f\|_{L_1(10 Q_u)} + \|f-l_u\|_{L_1(10 Q_u)} \lesssim \lambda \frac{\delta_z(Q_w)}{\delta_x(Q_w)} |Q_w|.$$
  Since $Q_w\subset 10Q_u$ and since $l_w$ and $l_u$ are affine functions, 
  \begin{equation}\label{eq:approximating plane differences}
    \|l_w-l_u\|_{L_\infty(Q_w)}\le \|l_w-l_u\|_{L_\infty(10 Q_u)} \lesssim \lambda \frac{\delta_z(Q_w)}{\delta_x(Q_w)}.
  \end{equation}

  Let $u$ and $u'$ be the children of $w$.  Then
  $$\|g_S-g_{S'}\|_{L_\infty(Q_w)}\le \max\{\|l_w-l_u\|_{L_\infty(Q_w)}, \|l_w-l_{u'}\|_{L_\infty(Q_w)}\} \lesssim \lambda \frac{\delta_z(Q_w)}{\delta_x(Q_w)},$$
  as desired.
\end{proof}

Next, we consider adding coherent subsets to $S$.
For a horizontally-cut vertex $w\in \cVh(\Delta)$ and a descendant $v\in \cD(w)$, we say that $v$ is an \emph{$h$--descendant} of $w$ if every vertex on the path from $w$ to $v$, except possibly $v$ itself, is horizontally cut. We let the \emph{$h$--subtree} $D^h(w)\subset \cD(w)$ to be the set of $h$--descendants of $w$. Note that $D^h(w)$ is coherent, and it corresponds to a partition of $Q_w$ into a stack of pseudoquads, all with the same width as $Q_w$.

\begin{lemma}\label{lem:horizontal differences}
  Let $S\subset \cV(\Delta)$ be a coherent set and let $w\in \mathsf{min}(S)\cap \cVh(\Delta)$. Let $S'=S\cup D^h(w)$. Then $\supp(g_S-g_{S'})\subset Q_w$ and
  $$\|g_{S'}-g_S\|_{L_\infty(Q_w)} \lesssim \lambda \frac{\delta_z(Q_w)}{\delta_x(Q_w)}.$$
\end{lemma}
\begin{proof}
  It follows from the definitions that $g_S$ and $g_{S'}$ agree outside of $Q_w$. 
  For $i\ge 0$, let $D^h_{i}(w) =D^h(w) \cap \cC^{i}(w)$,
  $D^h_{\le i}(w) = D^h(w) \cap \cC^{\le i}(w)$, and $h_i=g_{S\cup D^h_i(w)}$. Then $h_0=g_{S}$. By Lemma~\ref{lem:coherent convergence}, $h_i$ converges to $g_S$ pointwise almost everywhere.

  Every pseudoquad in $D^h(w)$ has width $\delta_x(Q_w)$, and by Lemma~\ref{lem:soft cut props}, for any $i\ge 0$ and $v\in D^h_i(w)$, 
  $$\delta_z(Q_v)\le \left(\frac{3}{4}\right)^i \delta_z(Q_w).$$
  By Lemma~\ref{lem:child differences}, 
  $$\|h_i - h_{i+1} \|_{L_\infty(Q)} \lesssim \lambda \left(\frac{3}{4}\right)^i \frac{\delta_z(Q_w)}{\delta_x(Q_w)}.$$
  Thus
  $$\|g_{S}- g_{S'}\|_{L_\infty(Q)} \le \lim_i \|h_i - h_0\|_{L_\infty(Q)} \lesssim \sum_{i=0}^\infty \lambda \left(\frac{3}{4}\right)^i \frac{\delta_z(Q_w)}{\delta_x(Q_w)}.$$
\end{proof}

We will use these bounds let us construct a sequence of approximations of $f$ that converge in $L_p$.
For $v\in \cV(\Delta)$ and $j\ge 0$, let
\begin{equation}\label{eq:def-Rj}
  R_j(v)=\left\{w\in \cD(v)\mid \delta_x(Q_w)\ge 2^{-j}\delta_x(Q_v) \right\}.
\end{equation}
By Lemma~\ref{lem:coherent convergence}, the $g_{R_j(v)}$ converge to $f$ in $L_1(Q_v)$; we will show that they converge in $L_p(Q_v)$ too.

We will need to estimate the weight of the descendants of $v$. Let $P_j(v)=\min R_j(v)$.  Then
$$P_j(v)=\left\{w\in \cDv(v)\mid \delta_x(Q_w)=2^{-j}\delta_x(Q_v) \right\}.$$

The following Vitali-type covering lemma is proved in \cite{NY-Foliated}. (In \cite{NY-Foliated}, it is stated for $\frac{1}{32r^2}$--rectilinear foliated patchworks, but all PSFCs are $\frac{1}{32r^2}$--rectilinear.)
\begin{lemma}[{\cite[Lem.\ 9.4]{NY-Foliated}}]
  \label{lem:grid coverings}
  Let $\Delta$ be a PSFC. For any $j\ge 0$, there is a subset $V_j(v)\subset P_j(v)$ such that the sets $r Q_w, w\in V_j(v)$ are pairwise disjoint, and $W(V_j(v))\approx W(Q_j(v))$.  It follows that
  \begin{equation}\label{eq:grid coverings}
    W(\cDv(v)) \approx_r \sum_j W(V_j(v)).
  \end{equation}
\end{lemma}

This lets us bound $W(P_i(v))$ for $v\in \cVh(\Delta)$. 
\begin{cor}\label{cor:vitali weights}
  Let $\Delta$ be a PSFC and let $v\in \cVh(\Delta)$. For all $i\ge 0$,
  $$W(P_i(v)) \lesssim 2^i \alpha(Q_v)^{-4} |Q_v| = 2^i W(\{v\}).$$
\end{cor}
\begin{proof}
  For every $w\in P_i(v)$, the pseudoquad $Q_w$ is vertically cut, so $\Gamma$ is not $(\eta, R)$--paramonotone on $rQ_v$.  That is,
  $$\Omega^P_{\Gamma^+,2^{-i} R\delta_x(Q_v)}(r Q_w) \ge \eta \alpha(Q_w)^{-4}|Q_w|=\eta W(\{Q_w\}).$$

  Then, by Lemma~\ref{lem:grid coverings},
  $$W(P_i(v))\approx W(V_i(v)) \le \eta^{-1}\sum_{w\in V_i(v)} \Omega^P_{\Gamma^+,2^{-i} R\delta_x(Q_v)}(r Q_w).$$
  The $rQ_w$'s are disjoint and, by Lemma~\ref{lem:child-hierarchy}, they are contained in $r Q_v$, so 
  \begin{align*}
    W(P_i(v))
    &\lesssim \eta^{-1} \Omega^P_{\Gamma^+,2^{-i} R\delta_x(Q_v)}(r Q_v) \\ 
    &\stackrel{\eqref{eq:Omega length change}}{\le} 2^i \eta^{-1} \Omega^P_{\Gamma^+,R\delta_x(Q_v)}(r Q_v).
  \end{align*}

  Since $Q_v$ is horizontally cut, $\Gamma$ is $(\eta, R)$--paramonotone on $rQ_v$, so
  $$\Omega^P_{\Gamma^+,R\delta_x(Q_v)}(r Q_v) \le \eta \alpha(Q_v)^{-4}|Q_v|$$
  and thus
  $$W(P_i(v)) \lesssim 2^i \alpha(Q_v)^{-4}|Q_v|,$$
  as desired. 
\end{proof}

This leads to the desired $L_p$ bounds.
\begin{proof}[{Proof of Proposition~\ref{prop:Lp approximating planes}}]
  Let $\Delta$ be a PSFC and let $D_0$, $\mu$, $\lambda$, $\eta$, $R$, and $r$ be the parameters in Definition~\ref{def:coronization}. Let $2\le p < 5$.

  Let $v\in \cVh(\Delta)$.  Let $Q=Q_v$.  Let $w=\delta_x(Q)$, $h=\delta_z(Q)$.  (In fact, the problem has enough symmetry that it suffices to prove the proposition when $\delta_x(Q)=\delta_z(Q)=\alpha(Q)=1$.) For $i\ge 0$, let $R_i = \{u\in \cD(v)\mid \delta_x(Q_u)\ge 2^{-i}w\}$ as in \eqref{eq:def-Rj}. We consider the approximations $g_{R_i}$.

  By Lemma~\ref{lem:coherent convergence}, the functions $g_{R_i}$ converge pointwise to $f$, so by Fatou's Lemma,
  $$\|f-l_v\|_{L_p(Q)} \le \mathop{\lim \inf}_i \|g_{R_i}-l_v\|_{L_p(Q)} \\ \le \|g_{R_0}-l_v\|_{L_p(Q)} + \mathop{\lim \inf}_i \|g_{R_i}-g_{R_0}\|_{L_p(Q)}.$$
    
  First, note that $D^h(v) = R_0$, so by Lemma~\ref{lem:horizontal differences},
  $$\|g_{R_0}-l_{v}\|_{L_\infty(Q)} = \|g_{R_0}-g_{\{v\}}\|_{L_\infty(Q)} \lesssim \lambda \cdot \frac{h}{w},$$
  so
  $$\|g_{R_0}-l_v\|_{L_p(Q)} \lesssim \lambda^p \frac{h^p}{w^p}|Q|.$$

  Likewise, for every $u\in \mathsf{min}(R_i)$, Lemma~\ref{lem:horizontal differences} implies that
  $$\|g_{R_{i+1}}-g_{R_i}\|_{L_\infty(Q_u)} \lesssim \lambda \frac{\delta_z(Q_u)}{\delta_x(Q_u)}.$$
  Since $\delta_x(Q_u)=2^{-i}w$ for all $u\in \mathsf{min}(R_i)$,
  \begin{equation}\label{eq:step-bound}
    \|g_{R_{i+1}}-g_{R_i}\|_{L_p(Q)}^p \lesssim \sum_{u\in \mathsf{min}(R_i)} \lambda^p \frac{\delta_z(Q_u)^p}{\delta_x(Q_u)^p}|Q_u| \\ \approx \lambda^p w^{-p+1} 2^{i(p-1)}  \sum_{u\in \mathsf{min}(R_i)} \delta_z(Q_u)^{p+1}.
  \end{equation}

  We bound the $\delta_z(Q_u)$'s using Corollary~\ref{cor:vitali weights}. For all $i\ge 0$, \eqref{eq:weight deltax} implies that 
  $$W(\mathsf{min}(R_i)) \approx 2^{3i}w^{-3} \sum_{u\in \mathsf{min}(R_i)} \delta_z(Q_u)^3.$$
  Then, by Corollary~\ref{cor:vitali weights},
  $$\sum_{u\in \mathsf{min}(R_i)} \delta_z(Q_u)^3 \approx  2^{-3i}w^{3}  W(\mathsf{min}(R_i))  \lesssim 2^{-3i}w^3 \cdot 2^i \alpha(Q)^{-4} |Q| \stackrel{\eqref{eq:weight deltax}}{\approx} 2^{-2i} h^3.$$
  Since $p \ge 2$, convexity implies
  $$\sum_{u\in \mathsf{min}(R_i)} \delta_z(Q_u)^{p+1} \le \left(\sum_{u\in \mathsf{min}(R_i)} \delta_z(Q_u)^3\right)^{\frac{p+1}{3}}\lesssim 2^{-2i \cdot \frac{p+1}{3}} h^{p+1}.$$
  Then
  \begin{align*}
    \|g_{R_{i+1}}-g_{R_i}\|_{L_p(Q)}^p 
    &\lesssim \lambda^p  w^{-p+1} 2^{i(p-1)} \cdot 2^{-2i \cdot \frac{p+1}{3}} h^{p+1} \\
    &=\lambda^p w^{-p+1} h^{p+1}  2^{\frac{ip}{3}-\frac{5i}{3}},
  \end{align*}
  so by \eqref{eq:step-bound},
  $$\|g_{R_{i+1}}-g_{R_i}\|_{L_p(Q)} \lesssim \lambda 2^{i\cdot \frac{p-5}{3p}} \frac{\delta_z(Q)}{\delta_x(Q)}|Q|^{\frac{1}{p}}.$$
  Since $p<5$, this decays exponentially in $i$, so for any $n$,
  \begin{align*}
    \|g_{R_n}-l_v\|_{L_p(Q)} 
    &\le \|g_{R_0}-l_v\|_{L_p(Q)}  + \sum_{i=0}^j \|g_{R_{i+1}}-g_{R_i}\|_{L_p(Q)} \\ 
    &\lesssim \frac{1}{1-2^{\frac{p-5}{3}}} \cdot \lambda \cdot \frac{\delta_z(Q)}{\delta_x(Q)} |Q|^{\frac{1}{p}} \\
    & \approx \frac{\lambda}{5-p}\cdot\frac{\delta_z(Q)}{\delta_x(Q)} |Q|^{\frac{1}{p}}.
  \end{align*}
  Therefore, when $v\in \cVh(\Delta)$,
  $$\|f-l_v\|_{L_p(Q)} \lesssim \frac{1}{5-p} \lambda \frac{\delta_z(Q)}{\delta_x(Q)} |Q|^{\frac{1}{p}}.$$

  When $v\in \cVv(\Delta)$, we proceed similarly to Lemma~\ref{lem:full approximating planes}. Let $m\in \cVh(\Delta)$ be the minimal horizontally-cut ancestor of $v$. Then $l_m=l_v$, $\delta_z(Q_m)\approx \delta_z(Q_v)$, and $\delta_x(Q_m)\gtrsim \delta_x(Q_v)$. By the horizontally-cut case,
  $$\|f-l_m\|_{L_p(Q_m)} \lesssim \lambda \frac{1}{5-p} \frac{\delta_z(Q_m)}{\delta_x(Q_m)} |Q_m|^{\frac{1}{p}}.$$
  Since $p\ge 1$, we have
  $$\frac{\delta_z(Q_m)}{\delta_x(Q_m)} |Q_m|^{\frac{1}{p}} \approx \frac{\delta_z(Q_m)^{\frac{p+1}{p}}}{\delta_x(Q_m)^{\frac{p-1}{p}}} 
  \lesssim \frac{\delta_z(Q_v)}{\delta_x(Q_v)} |Q_v|^{\frac{1}{p}},$$
  so
  $$\|l_v-f\|_{L_p(Q_v)}
    \le \|l_m-f\|_{L_p(Q_m)} 
    \lesssim \lambda \frac{1}{5-p} \frac{\delta_z(Q_v)}{\delta_x(Q_v)} |Q_v|^{\frac{1}{p}}.$$
  This proves the proposition for $2\le p < 5$. If $p=1+\theta$ for $\theta\in (0,1)$, then by Lyapunov's inequality and \eqref{eq:def lambda approximating},
  \begin{multline*}
    \|l_v-f\|_{L_p(Q_v)} \le \|l_v-f\|_{L_1(Q_v)}^{\frac{1-\theta}{p}}\|l_v-f\|_{L_2(Q_v)}^{\frac{2\theta}{p}}\\
    \lesssim \left(\lambda \frac{\delta_z(Q_v)}{\delta_x(Q_v)} |Q_v|\right)^{\frac{1-\theta}{p}}\left(\lambda \frac{\delta_z(Q_v)}{\delta_x(Q_v)} |Q_v|^{\frac{1}{2}}\right)^{\frac{2\theta}{p}} = 
    \lambda \frac{\delta_z(Q_v)}{\delta_x(Q_v)} |Q_v|^{\frac{1}{p}},
  \end{multline*}
  so the proposition holds for $p\in [1,5)$.
\end{proof}

\section{Proof of Proposition \ref{prop:Q-carleson}}\label{sec:proof}

In this section, we will prove Proposition~\ref{prop:Q-carleson}. Recall that the proposition states that there is $\tau>0$ depending only on the intrinsic Lipschitz constant $L$ of $\Gamma=\Gamma_f$ such that if $\alpha_0$ is sufficiently small (again depending on $L$) and $Q$ is a rectilinear pseudoquad with $\alpha(Q) = \alpha_0$, then
\begin{equation}\label{eq:prop-repeat}
  \int_{\frac{1}{3}Q} \int_0^{\tau \delta_x(Q)} \gamma_{4,f}(x,r)^4 \frac{\ud r}{r} \ud x \lesssim |Q|
\end{equation}
where $\gamma_{4,f}(x,r) = r^{-\frac{7}{4}} \inf_{h\in \Aff} \|f - h \|_{L_4(V(\Psi_f(x),r))}$.
Most of the implicit constants in this section depend on the intrinsic Lipschitz constant of $\Gamma$, so we will write $\lesssim$ for $\lesssim_L$ and so on for brevity.


We take $\alpha_{\mathrm{min}}$ as in Lemma~\ref{lem:coronizations-exist}, let $0<\alpha_0<\alpha_{\mathrm{min}}$, and let $Q$ be a rectilinear pseudoquad with $\alpha(Q)=\alpha_0$. Since the conclusion of Proposition \ref{prop:Q-carleson} is invariant under scaling, we rescale so that  $\delta_z(Q)=1$ and $\delta_x(Q)=\alpha(Q)<1$. By Lemma~\ref{lem:coronizations-exist}, there is a PSFC with root $Q$, which we call $\Delta$. 

We prove Proposition~\ref{prop:Q-carleson} by using this PSFC to construct a sequence of approximations of $f$.
For $i\ge 0$, let 
\begin{equation}\label{eq:def-Si}
  S_i:=\{w\in \cV(\Delta)\mid \delta_z(Q_w)\ge 2^{-2i}\}.
\end{equation}
If $w$ and $w'$ are siblings, then $\delta_z(Q_w)= \delta_z(Q_{w'})$, and if $v$ is the parent of $w$, then $\delta_z(Q_v) \ge \delta_z(Q_{w})$, so $S_i$ is a coherent set. 
Furthermore, by Lemma~\ref{lem:aspect-bounds}, we have
\begin{equation*}
  \delta_x(Q_v)\ge \frac{\alpha_{\mathrm{min}}}{4} \sqrt{\delta_z(Q_v)} \ge \frac{\alpha_{\mathrm{min}}}{4} 2^{-i}
\end{equation*}
for all but finitely many $v\in S_i$, so $S_i$ is finite. We let $F_i:=\min(S_i)$; this is a partition of $Q$. By Lemma~\ref{lem:soft cut props}, if $\delta_z(Q_v)\ge 4\cdot2^{-2i}$, then $Q_v$'s children have height greater than $2^{-2i}$, so $v\not \in F_i$. That is,  
\begin{equation}\label{eq:Fi-heights}
  \delta_z(Q_w)\in [2^{-2i},4\cdot 2^{-2i}) \text{ for all $w\in F_i$.}
\end{equation}
For each $i$, let $g_{S_i}$ be the approximation defined in Definition~\ref{def:gS}, which agrees with $l_w$ on $Q_w$ for each $w\in F_i$.

For $x\in V_0$, let 
$$\sigma_i(x,r) = r^{-\frac{7}{4}} \inf_{h\in \Aff} \|g_{S_i} - h \|_{L_4(V(\Psi_f(x),r))}.$$
This is similar to $\gamma_{4,g_{S_i}}$ except that the norm is taken in $L_4(V(\Psi_f(x),r))$. By the triangle inequality, for any $x$ and $r$ such that $V(\Psi_f(x),r)\subset Q$,
\begin{equation}\label{eq:gamma-sigma}
  \gamma_{4,f}(x,r) \le r^{-\frac{7}{4}} \|g_{S_i} - f\|_{L_4(V(\Psi_f(x),r))} + \sigma_i(x,r).
\end{equation}

We will first state and prove bounds on the terms above, then prove Proposition~\ref{prop:Q-carleson}.
We can bound $g_{S_i} - f$ using Proposition~\ref{prop:Lp approximating planes}.
Since the $F_i$ are minimal elements of $S_i$, every element of $F_i$ is horizontally cut. Therefore, by Proposition~\ref{prop:Lp approximating planes}, 
\begin{equation}\label{eq:L4-approx}
  \|l_w - f\|_{L_4(Q_w)}\lesssim \frac{\delta_z(Q_w)}{\delta_x(Q_w)} |Q_w|^{\frac{1}{4}}
\end{equation}
for all $w\in F_i$.  Then, since  $\delta_z(Q_w)\approx 2^{-2i}$,
\begin{equation}\label{eq:L4-weights}
  \left(\frac{\|g_{S_i} - f\|_{L_4(Q)}}{2^{-i}}\right)^4 \lesssim \sum_{w\in F_i} \frac{2^{4i}\delta_z(Q_w)^4}{\delta_x(Q_w)^4} |Q_w|\approx W(F_i),
\end{equation}
where $W(F_i) = \sum_{w\in F_i} \alpha(Q_w)^{-4}|Q_w|$ is as in Definition~\ref{def:weighted Carleson}.

We can use the following lemma to bound $\sigma_i(x,r)$.
\begin{lemma} \label{lem:gamma-Si-weight}
  There exists a constant $\nu > 0$ depending only on $L$ such that $V(\Psi_f(x),\nu)\subset Q$ for all $x\in \frac{1}{3}Q$ and such that for any $i\ge 0$ such that $2^{-i}\le \frac{\delta_x(Q)}{\alpha_{\mathrm{min}}}$, we have
  \begin{align}
    \int_{\frac{1}{3}Q} \sigma_{i}\big(x, \nu 2^{-i}\big)^4 \ud x \lesssim W(F_i). \label{eq:gamma-Si-weight}
  \end{align}
\end{lemma}
These two bounds will imply the proposition.

Before we prove Lemma~\ref{lem:gamma-Si-weight}, we state some lemmas that we will prove in Section~\ref{sec:pq-lemmas}. 

\begin{lemma} \label{lem:V-contained}
  For any \(0 < L < 1\) and \(0 < a < b \leq 4\), there exists \(0 < \eta < 1\) so that the following holds.  Let $\Gamma=\Gamma_f$ be an $L$--intrinsic Lipschitz graph, and let $Q$ be a rectilinear pseudoquad for $\Gamma$. Let $r = \min\{\sqrt{\delta_z(Q)}, \delta_x(Q)\}$. If \(V(p,\eta r)\) intersects $aQ$ for some $p \in \Gamma$ then $V(p,\eta r) \subseteq bQ$.
\end{lemma}

\begin{lemma}\label{lem:nearby-contain}
  Let $v,w \in F_i$.  If \(\delta_x(Q_w) \leq \delta_x(Q_v)\) and $Q_w \cap 3Q_v \neq \emptyset$, then $Q_w \subset 10 Q_v$.
\end{lemma}

This lemma lets us prove the following bound.
\begin{lemma}\label{lem:nearby-diff}
  Let $v,w \in F_i$.  If \(\delta_x(Q_w) \leq \delta_x(Q_v)\) and $Q_w \cap 3Q_v \neq \emptyset$, then 
    \begin{align}
    \|l_{w} - l_{v}\|_{L_\infty(Q_w)} \lesssim \frac{\delta_z(Q_w)}{\delta_x(Q_w)}. \label{eq:affine-diff}
  \end{align}
\end{lemma}
\begin{proof}
  By Lemma~\ref{lem:nearby-contain}, $Q_w \subset 10Q_v$, so 
  \begin{align*}
    \frac{1}{|Q_w|} \|f - l_{Q_v}\|_{L_1(Q_w)} \leq \frac{1}{|Q_w|} \|f - l_{Q_v}\|_{L_1(10Q_v)} \overset{\eqref{eq:L1-bound}}{\leq} \frac{1}{|Q_w|} \lambda \frac{\delta_z(Q_v)}{\delta_x(Q_v)} |Q_v| \lesssim \frac{\delta_z(Q_w)}{\delta_x(Q_w)}.
  \end{align*}
  In the last inequality, we used the fact that $\delta_z(Q_v) \approx \delta_z(Q_w)$ and so Lemma~\ref{lem:soft cut props} gives that $|Q_v|/|Q_w| \approx \delta_x(Q_v)/\delta_x(Q_w)$.

  Likewise, 
  \begin{align*}
    \frac{1}{|Q_w|} \|f - l_{Q_w}\|_{L_1(Q_w)}  \overset{\eqref{eq:L1-bound}}{\leq} \lambda \frac{\delta_z(Q_w)}{\delta_x(Q_w)},
  \end{align*}
  so by the triangle inequality,
  \begin{align*}
    \frac{1}{|Q_w|}\|l_{Q_w} - l_{Q_v}\|_{L_1(Q_w)}  \lesssim  \frac{\delta_z(Q_w)}{\delta_x(Q_w)}.
  \end{align*}
  
  It remains to show that $|Q_w| \|l_v-l_w\|_{L_\infty(Q_w)}\lesssim \|l_v-l_w\|_{L_1(Q_w)}$. In fact, we claim that if $L$ is any affine function and $Q$ is a rectilinear pseudoquad, then 
  \begin{align}
    \|L\|_{L_\infty(Q)} \leq 24 \frac{\|L\|_{L_1(Q)}}{|Q|}. \label{eq:affine-lp-comp}
  \end{align}
  
  Let $I=x(Q)$ be the base of $Q$. After a scaling and translation, we may suppose that $I=[-1,1]$. Since $Q$ is rectilinear, there are functions $g_1,g_2\from I\to \R$ such that $Q=\{(x,0,z) | x\in I, z\in [g_1(x),g_2(x)]\}$ and quadratic polynomials $h_1,h_2\from I\to \R$ such that $h_2=h_1+\delta_z(Q)$ and $\|g_i-h_i\|_\infty \le \frac{1}{32}\delta_z(Q)$. In particular, $|Q|\le 3 \delta_z(Q)$.
  
  Let $M=\|L\|_{L_\infty(Q)}$. Since $L(x,0,z)= ax + b$ for some $a$ and $b$, we have $M = |L(1,0,0)|$ or $M = |L(-1,0,0)|$; suppose $M=|L(1,0,0)|$. (The other case is similar.)
  
  Then $|L(x,0,0)| \ge \frac{M}{2}$ for $x\ge \frac{1}{2}$, so
  $$\|L\|_{L_1(Q)}\ge \int_{\frac{1}{2}}^1 \int_{g_1(x)}^{g_2(x)} \frac{M}{2}\ud z \ud x \ge \frac{M}{8}\delta_z(Q) \ge \frac{M}{24}|Q|.$$
  This proves the lemma.
\end{proof}

Now we prove Lemma~\ref{lem:gamma-Si-weight}.
\begin{proof}[Proof of Lemma~\ref{lem:gamma-Si-weight}]
  Let $\zeta\from Q\to \R$ be the function such that for any $w\in F_i$ and $x\in Q_w$, $\zeta(x) = \frac{\delta_z(Q_w)}{\delta_x(Q_w)}$. If $w\ne w'$, then $Q_w$ and $Q_{w'}$ intersect in a set of measure zero, so we break ties arbitrarily. We will bound $\gamma_{4,g_{S_i}}$ in terms of $\|\zeta\|_4$.
  
  For all $v\in F_i$, we have $\delta_z(Q_v)\in [2^{-2i}, 4\cdot 2^{-2i}]$. By Lemma~\ref{lem:aspect-bounds}, this implies $\alpha(Q_v)\ge \frac{\alpha_{\mathrm{min}}}{4}$, so $\delta_x(Q_v)\ge \frac{\alpha_{\mathrm{min}}}{4}2^{-i}$. Then 
  $$\frac{\alpha_{\mathrm{min}}}{4} 2^{-i} \le \min\{\delta_x(Q_v),\sqrt{\delta_z(Q_v)}\}.$$
  By Lemma~\ref{lem:Q-2Q} and Lemma~\ref{lem:V-contained}, there is an $0 < \eta_0 < 1$ such that if $v\in F_i$, $p\in \Gamma$, and $Q_v\cap V(p, \frac{\alpha_{\mathrm{min}}}{4}  \eta_0 2^{-i})\ne \emptyset$, then $V(p, \frac{\alpha_{\mathrm{min}}}{4} \eta_0 2^{-i})\subset 3Q_v$. Likewise, there is a $0 < \eta_1 < 1$ such that if $x\in \frac{1}{3}Q_v$, then 
  $$V(\Psi_f(x), \frac{\alpha_{\mathrm{min}}}{4} \eta_1 2^{-i})\subset \frac{2}{3}Q_v\subset Q_v.$$
  Let $\nu = \frac{\alpha_{\mathrm{min}}}{4}\min\{\eta_0,\eta_1\}$; we can choose $\nu$ to depend only on $L$. Then $V(\Psi_f(x),\nu)\subset Q$ for all $x\in \frac{1}{3}Q$.

  Let $x\in \frac{1}{3}Q$ and let $D(x)=V(\Psi_f(x),\nu 2^{-i})$. Let $v_1,\dots, v_n\in F_i$ be the elements such that $Q_{v_i}$ intersects $D(x)$ and suppose that $\delta_x(Q_{v_1})\ge \delta_x(Q_{v_i})$ for all $i$. For any $y \in D(x)$, we have $y\in 3 Q_{v_1}$ and there is a $j$ such that $y\in Q_{v_j}$. By Lemma~\ref{lem:nearby-diff}, 
  $$|g_{S_i}(y) - l_{v_1}(y)|=|l_{v_j}(y) - l_{v_1}(y)| \lesssim \zeta(y).$$
  Therefore,
  \begin{equation}\label{eq:zeta-bound}
    \sigma_i(x,\nu 2^{-i}) \le (\nu 2^{-i})^{-\frac{7}{4}} \|g_{S_i} - l_{v_1} \|_{L_4(D(x))} \lesssim 2^{\frac{7i}{4}} \|\zeta\|_{L_4(D(x))}
  \end{equation}
  and
  $$\int_{\frac{1}{3}Q}\sigma_i(x, \nu 2^{-i})^4 \ud x \lesssim 2^{7i} \int_{\frac{1}{3}Q} \int_{D(x)} \zeta^4(y) \ud y \ud x =: I.$$
  
  Let $c$ be as in Lemma~\ref{lem:V-comparable}.
  Then for every $x\in \frac{1}{3}Q$ and $y\in D(x)$, we have that $x\in V(\Psi_f(y), c \nu 2^{-i})$. Our choice of $\nu$ implies that $y\in Q$, so by Fubini's Theorem,
  $$I \le 2^{7i} \int_{Q} \int_{V(\Psi_f(y), c \nu 2^{-i})} \zeta^4(y) \ud x \ud y\lesssim 2^{4i} \|\zeta\|^4_{L_4(Q)}.$$
  
  Furthermore, since $\delta_z(Q_w)\approx 2^{-2i}$ for all $w\in F_i$,
  $$\|\zeta\|^4_{L_4(Q)} = \sum_{w\in F_i} |Q_w| \frac{\delta_z(Q_w)^4}{\delta_x(Q_w)^4} = \sum_{w\in F_i} \frac{|Q_w| \delta_z(Q_w)^2}{\alpha^{4}(Q_w)} \approx 2^{-4i} W(F_i),$$
  so 
  $$\int_{\frac{1}{3}Q}\sigma_i(x, \nu 2^{-i})^4 \ud x \lesssim 2^{4i} \|\zeta\|^4_{L_4(Q)} \lesssim W(F_i)$$
  as desired.
\end{proof}

Finally, we prove Proposition~\ref{prop:Q-carleson}.
\begin{proof}[Proof of Proposition \ref{prop:Q-carleson}]
  Let $\Delta$, $S_i$, and $F_i$ be as above and let $\nu$ be as in Lemma~\ref{lem:gamma-Si-weight}. 
  Let $\tau = \frac{\nu}{2\alpha_{\mathrm{min}}}$ and $i_0 = \left\lceil\log_2\frac{\alpha_{\mathrm{min}}}{\delta_x(Q)}\right\rceil$
  so that $2^{-i_0} \le \frac{\delta_x(Q)}{\alpha_{\mathrm{min}}}$ and $\nu 2^{-i_0}\ge \tau \delta_x(Q)$.
  
  Suppose that $i\ge i_0$ and let $r_i=\nu 2^{-i}$. By \eqref{eq:gamma-sigma} and Lemma~\ref{lem:gamma-Si-weight},
  $$\int_{\frac{1}{3} Q} \gamma_{4,f}(x,r_i)^4 \ud x \lesssim \int_{\frac{1}{3} Q} r_i^{-7} \|g_{S_i} - f\|^4_{L_4(V(\Psi_f(x),r_i))}\ud x + W(F_i).$$ 
  As in the proof of Lemma~\ref{lem:gamma-Si-weight}, Fubini's theorem shows
  $$r_i^{-7}\int_{\frac{1}{3} Q} \|g_{S_i} - f\|^4_{L_4(V(\Psi_f(x),r_i))}\ud x \lesssim  r_i^{-4} \|g_{S_i}(y)-f(y))\|_{L_4(Q)}^4,$$
  and by \eqref{eq:L4-weights},
  $$\int_{\frac{1}{3} Q} \gamma_{4,f}(x,r_i)^4 \ud x \lesssim W(F_i) + W(F_i)\lesssim W(F_i).$$ 
  
  For $r\in [\frac{r_i}{2},r]$, we have $\gamma_{4,f}(x,r)\lesssim  \gamma_{4,f}(x,r_i)$, so
  $$\int_{\frac{1}{3} Q} \int_{\frac{r_i}{2}}^{r_i}\gamma_{4,f}(x,r)^4\frac{\ud r}{r} \ud x \lesssim \int_{\frac{1}{3} Q} \gamma_{4,f}(x,r)^4 \ud x \lesssim W(F_i).$$
  By \eqref{eq:Fi-heights}, if $w\in F_i$, then $\delta_z(Q_w)\in [2^{-2i},4\cdot 2^{-2i})$. It follows that the $F_i$'s are disjoint, so we can sum over $i\ge i_0$ to obtain
  $$\int_{\frac{1}{3} Q} \int_{0}^{\nu 2^{-i_0}} \gamma_{4,f}(x,r)^4\frac{\ud r}{r} \ud x \lesssim \sum_{i=i_0}^\infty W(F_i) \le  W(\cV(\Delta))\lesssim |Q|.$$
  This proves the proposition.
\end{proof}

\subsection{Geometry of rectilinear pseudoquads}
\label{sec:pq-lemmas}

In this section, we will prove some basic results and bounds for rectilinear pseudoquads of intrinsic Lipschitz graphs. Recall that a rectilinear pseudoquad for $\Gamma = \Gamma_f$ is a tuple $(Q,R_Q)$ of a pseudoquad $Q$ and a parabolic rectangle $R_Q$ approximating $Q$. Then $R_Q$ is a pseudoquad in a vertical plane of slope $\slope(Q)$, which we denote $P_Q$.  

Our first set of results deals with arbitrary rectilinear pseudoquads, without any assumption on paramonotonicity; our goal is to prove Lemma~\ref{lem:V-contained}. We first bound the slope of a rectilinear pseudoquad in terms of the intrinsic Lipschitz constant.

\begin{lemma} \label{lem:parabolic-bound}
  Let $(Q,R)$ be a rectilinear pseudoquad of an $L$--intrinsic Lipschitz graph. Then 
  \begin{align}
    |\slope Q| \leq \frac{L}{\sqrt{1-L^2}} + \frac{\delta_z(Q)}{\delta_x(Q)^2} = \frac{L}{\sqrt{1-L^2}} + \alpha(Q)^{-2}. \label{eq:parabolic-bound}
  \end{align}
\end{lemma}
\begin{proof}
  Let $I$ denote the base of $Q$.  By translation and scaling, we may assume $I = [-\delta_x(Q)/2,\delta_x(Q)/2]$. Let $W=2\delta_x(Q)$, so that $4I=[-W,W]$. Let $g$ and $h$ be the functions whose graphs are the characteristic curves forming the top boundaries of $Q$ and $R$, respectively. Since $h$ is quadratic and $\slope Q = -h''(x)$, there are $b,c\in \R$ such that $h(x) = -\slope(Q)\frac{x^2}{2} +bx + c$. Therefore,
  $$|\slope(Q)| = \frac{|h(-W) - 2 h(0) + h(W)|}{W^2}.$$
  
  By Lemma~\ref{lem:intrinsic-lipschitz-cc},
  \begin{align*}
    |g(-W) - & 2 g(0) + g(W)|\\ 
    & \le |g(-W) - g(0) + g'(0) W| + |g(W) - g(0) - g'(0) W| \\
    & \overset{\eqref{eq:g-linearize}}{\leq} \frac{L}{\sqrt{1-L^2}}W^2. 
  \end{align*}
  By the rectilinearity of $Q$,
  $$\max_{t \in [-W,W]} |g(t) - h(t)| \leq \frac{1}{32} \delta_z(Q),$$
  so
  $$|\slope(Q)| = \frac{|h(-W) - 2 h(0) + h(W)|}{W^2} \le \frac{L}{\sqrt{1-L^2}} + \frac{1}{32}\frac{\delta_z(Q)}{\delta_x(Q)^2}.$$
\end{proof}

We use this to bound the distance from $P_Q$ to $\Gamma_f$. 

\begin{lemma} \label{lem:pseudoquad-linfty}
  For any \(0 < L < 1\) there exists an $\eta > 0$ so that the following holds.  Let $\Gamma_f$ be an $L$--intrinsic Lipschitz graph and let $Q$ be a rectilinear pseudoquad for $\Gamma_f$. Let $\lambda$ be the affine function such that $P_Q=\Gamma_\lambda$.  Then for all $p\in 4Q$,
  \begin{equation}\label{eq:4Q-close-to-affine}
    |f(p) - \lambda(p)| \leq \eta \max\left\{\sqrt{\delta_z(Q)}, \frac{\delta_z(Q)}{\delta_x(Q)} \right\}.
  \end{equation}
\end{lemma}

\begin{proof}
  Let $I$ be the base of $Q$, and let $g_Q,g_R \from \R \to \R$ be functions whose graphs are the characteristic curves going through the tops of $Q$ and $R$, respectively. After a left translation, we may suppose that $I = [-\delta_x(Q)/2,\delta_x(Q)/2]$ and that $P_Q$ goes through $\mathbf{0}$. Then $\lambda(x,z) = mx$ and $g_R(x) = g_R(0) - \frac{1}{2}mx^2$, where $m=\slope(Q)$.
  
  We first show that if $\tau = \frac{3L}{\sqrt{1-L^2}} + 2$, then
  \begin{align}
    \max_{t \in 4I} |f(t,0,g_Q(t)) - mt| \leq \tau \max\left\{\sqrt{\delta_z(Q)}, \frac{\delta_z(Q)}{\delta_x(Q)}\right\}. \label{eq:boundary-small}
  \end{align}
  Let $s \in 4I$, $q = (s,0,g_1(s))$, and let $h=\min\{\sqrt{\delta_z(Q)},\delta_x(Q)\}$. Without loss of generality, suppose that $s + h\in 4I$. By rectilinearity, $|g_Q(x)-g_R(x)|\le \frac{1}{32}\delta_z(Q)$ for all $x\in 4I$. By Lemma~\ref{lem:intrinsic-lipschitz-cc}, 
  $$|g_Q(s+h) - g_Q(s) - h g_Q'(s)|\le \frac{L}{\sqrt{1-L^2}} \frac{h^2}{2},$$
  and by Lemma~\ref{lem:parabolic-bound},
  $$|g_R(s+h) - g_R(s) - h g_R'(s)| = \frac{1}{2} |m| h^2 \le \frac{2L}{\sqrt{1-L^2}}h^2 + \frac{\delta_z(Q)}{\delta_x(Q)^2} h^2.$$
  Since $|(g_R(s+h) - g_R(s)) - (g_Q(s+h) - g_Q(s))|\le \frac{1}{16}\delta_z(Q)$, these imply
  $$|h g_Q'(s) - h g_R'(s)| \le \frac{3L}{\sqrt{1-L^2}}h^2 + \frac{\delta_z(Q)}{\delta_x(Q)^2}h^2+ \frac{1}{16}\delta_z(Q).$$
  We divide both sides by $h$ to find
  \begin{align*}
    |g_Q'(s) + ms| & \le \frac{3L}{\sqrt{1-L^2}}\sqrt{\delta_z(Q)} + \frac{\delta_z(Q)}{\delta_x(Q)} + \max \left\{\sqrt{\delta_z(Q)}, \frac{\delta_z(Q)}{\delta_x(Q)}\right\}\\
    & \le \tau\max \left\{\sqrt{\delta_z(Q)}, \frac{\delta_z(Q)}{\delta_x(Q)}\right\}.
  \end{align*}
  By \eqref{eq:cc-pde}, $f(s,0,g_Q(s))=-g_Q'(s)$, which proves \eqref{eq:boundary-small}.

  Finally, let $(x,0,z) \in 4Q$.  Then $x \in 4I$ and $|z - g_Q(x)| \leq 16 \delta_z(Q)$, so
  \begin{align*}
    |f(x,0,z) - mx| &\leq |f(x,0,z) - f(x,0,g_Q(x))| + |f(x,0,g_Q(x)) - mx| \\
    &\overset{\eqref{eq:lipschitz-vertical} \wedge \eqref{eq:boundary-small}}{\leq} \frac{16}{1-L} \sqrt{\delta_z(Q)} + \tau \max\left\{\sqrt{\delta_z(Q)}, \frac{\delta_z(Q)}{\delta_x(Q)}\right\}.
  \end{align*}
  The lemma follows by taking $\eta = \tau + \frac{16}{1-L}$.
\end{proof}





Finally, we can prove Lemma \ref{lem:V-contained}.

\begin{proof}[Proof of Lemma \ref{lem:V-contained}]
  Let $x\in \Gamma$ and $\rho > 0$ so that $V(x,\rho)$ intersects $aQ$. Let $p\in V(x,\rho)\cap aQ$. By Lemma~\ref{lem:V-comparable}, there is a $c>0$ such that 
  \begin{equation}\label{eq:V-contained-sub}
    V(x,\rho) \subset \Pi(B(x,c \rho) \cap \Gamma) \subset \Pi(B(\Psi_f(p),2c \rho) \cap \Gamma) \subset V(p,2c \rho).
  \end{equation}
  We claim that there is some $\eta' > 0$ such that $V(p,\eta'r)\subset bQ$ for every $p\in \Psi_f(a Q)$. The full claim then follows by taking $\eta = (2c)^{-1}\eta'$.
  
  
  Let $I$ be the base of $Q$ and let $m=\slope(Q)$. After a left translation, we may suppose that $I = [-\delta_x(Q)/2,\delta_x(Q)/2]$, that $P_Q = \Gamma_\lambda$, where $\lambda(x) = m x$, and that $$R_Q=\left\{(x,0,z):x\in I, \Big|z + \frac{m}{2} x^2\Big|\le \frac{\delta_z(Q)}{2}\right\}.$$
  
  Let $v = (x_0, 0, z_0) \in aQ$. Then 
  $$\left|z_0 + \frac{m}{2} x_0^2\right|\le \frac{a^2}{2}\delta_z(Q).$$
  Let  $c = z_0 + \frac{m}{2} x_0^2$ and let $g(x) = c - \frac{m}{2} x^2$, 
  so that the graph of $g$ is the characteristic curve of $P_Q$  through $v$.
  Let $d = \frac{b^2}{2}-\frac{a^2}{2}$ and let
  $$T := \{ (x,0,z) | x\in bI, |z - g(x)| \le d\delta_z(Q)\}.$$
  If $(x,0,z)\in T$, then 
  $$\left|z + \frac{m}{2}x^2\right| \le |c| + d\delta_z(Q) \le\frac{b^2}{2} \delta_z(Q),$$
  so $T\subset bQ$. We claim that there is a $\nu$ depending only on $L$ such that $V(p,\nu r)\subset T$. 
  
  By Lemma \ref{lem:proj-quad}, 
  $$V(p,\nu r) \subset \left\{ (x,0,z) | |x-x_0|\le \nu r, |z - z_0 + (x-x_0) f(v)| \le (\nu r)^2\right\}.$$
  If $(x,0,z)\in V(p,\nu r)$, then
  \begin{equation}\label{eq:z-gx}
    |z - g(x)| \le (\nu r)^2 + |g(x) - z_0 + (x-x_0) f(v)| \le \nu^2\delta_z(Q) + |\ell(x)|,
  \end{equation}
  where $\ell(x) = g(x) - z_0 + (x-x_0) f(v)$. Note that $\ell(x_0) = 0$. By  Lemma~\ref{lem:pseudoquad-linfty},
  $$|\ell'(x_0)| = |g'(x_0) + f(v)| \overset{\eqref{eq:cc-pde}}{=} |f(v)-\lambda(v)| \le \eta \max\left \{\sqrt{\delta_z(Q)}, \frac{\delta_z(Q)}{\delta_x(Q)}\right\}=\eta \frac{\delta_z(Q)}{r}.$$
  Finally, $\ell''(x_0) = -m$. Therefore, for $|x-x_0|\le \nu r$,
  \begin{equation}\label{eq:ell-bound}
    |\ell(x)| \le \nu r \eta \frac{\delta_z(Q)}{r} + (\nu r)^2\frac{|m|}{2} \overset{\eqref{eq:parabolic-bound}}{\le}\nu \eta \delta_z(Q) + \nu^2 \delta_z(Q)\left(\frac{L}{\sqrt{1-L^2}}+1\right).
  \end{equation}
  If $0<\nu<b-a$ is sufficiently small and $(x,0,z)\in V(p,\nu r)$, then \eqref{eq:z-gx} and \eqref{eq:ell-bound} imply $|z-g(x)| \le d\delta_z(Q)$ and thus $(x,0,z)\in T$. Therefore, $V(p,\nu r)\subset bQ$, as desired.  


\end{proof}

Next, we consider pseudoquads that are part of a PSFC and prove Lemma~\ref{lem:nearby-contain}. Let $\Delta$ be a PSFC, let $S_i$ be as in \eqref{eq:def-Si}, and let $F_i=\min S_i$. For $p \in V_0$, let $\kappa_{\Gamma,p}$ be a function whose graph is a characteristic curve of $\Gamma$ going through $p$, i.e., a solution of \eqref{eq:cc-pde} with the initial condition $\kappa_{\Gamma,p}(x(p)) = z(p)$. For $w\in F_i$, let $R_w=R_{Q_w}$ and $P_w = P_{Q_w}$. Let $\rho_w$ be the affine function such that $R_w=\Gamma_{\rho_w}$. Since $Q_w$ is part of a PSFC, it has a $\lambda$--approximating plane $\Lambda_w$ that satisfies Proposition~\ref{prop:Omega-control}, and we let $l_w$ be the corresponding affine function. The construction of $P_w$ and $\Lambda_w$ are different, so they need not be the same plane, but they must be close, as in the following lemma.

\begin{lemma}\label{lem:two-close-planes}
  Let $\Delta$ be a PSFC and let $w\in \cV(\Delta)$. Let $I=x(Q_w)$. Then, for any $q\in 4Q_w$,
  $$\|\kappa_{P_w,q} - \kappa_{\Lambda_w,q}\|_{L_\infty(4I)} \leq \frac{3}{32} \delta_z(Q_w)$$
  and
  $$\|\kappa_{P_w,q} - \kappa_{\Gamma,q}\|_{L_\infty(4I)} \leq \frac{1}{8} \delta_z(Q_w).$$
\end{lemma}
Indeed, this lemma holds for any pseudoquad $Q$ that satisfies Proposition~\ref{prop:Omega-control}.
\begin{proof}
  Let $g_1,g_2,h_1,h_2\from 4 I\to \R$ be the functions parameterizing the boundary of $Q_w$, so that $Q_w = \{(x,0,z) | x\in I, z\in [g_1(x),g_2(x)]\}$ and $R_w = \{(x,0,z) | x\in I, z\in [h_1(x),h_2(x)]\}$. 
  
  Let $q\in 4Q_w$ and let $p = (x(q), 0, g_1(q))$ so that $p$ lies in the graph of $g_1$; we can take $g_1 = \kappa_{\Gamma,p}$. Since $\Lambda_w$ and $P_w$ are vertical planes,
  $$\kappa_{\Lambda_w,p} - \kappa_{P_w,p} = \kappa_{\Lambda_w,q} - \kappa_{P_w,q},$$
  and we have
  \begin{equation}\label{eq:two-close-one}
    \|\kappa_{P_w,p} - \kappa_{\Lambda_w,p}\|_{\infty} \leq \|\kappa_{P_w,p} - h_1\|_{\infty} + \|h_1 - g_1\|_{\infty} + \|g_1 - \kappa_{\Lambda_w,p}\|_{\infty}
  \end{equation}
  where all norms are in $L_\infty(4I)$. By the rectilinearity of $Q_w$, the first two terms are each at most $\frac{1}{32}\delta_z(Q)$; the last term is at most $\frac{1}{32}\delta_z(Q)$ by Proposition~\ref{prop:Omega-control}. This proves the first inequality.
  
  Thus, by \eqref{eq:two-close-one} and Proposition~\ref{prop:Omega-control},
  \begin{align*}
    \|\kappa_{P_w,q} - \kappa_{\Gamma,q}\|_{\infty} &\leq \|\kappa_{P_w,q} - \kappa_{\Lambda_w,q}\|_{\infty} + \|\kappa_{\Lambda_w,q} - \kappa_{\Gamma,q}\|_{\infty}\\
    &\le \frac{3}{32}\delta_z(Q_w) + \frac{1}{32}\delta_z(Q_w).
  \end{align*}
  This proves the lemma.
\end{proof}
Note that, by \eqref{eq:cc-pde}, this implies
$$\|\rho_w - l_w\|_{L_\infty(4Q_w)}\lesssim \frac{\delta_z(Q_w)}{\delta_x(Q_w)}.$$

We use Lemma~\ref{lem:two-close-planes} to prove Lemma~\ref{lem:nearby-contain}, which states that if $v,w \in F_i$, \(\delta_x(Q_w) \leq \delta_x(Q_v)\), and $Q_w \cap 3Q_v \neq \emptyset$, then $Q_w \subset 10 Q_v$.
\begin{proof}[Proof of Lemma~\ref{lem:nearby-contain}]
  Let $I_v=x(Q_v)$ and $I_w=x(Q_w)$ be the bases of $Q_v$ and $Q_w$. Since $Q_w\cap 3Q_v$ is nonempty, $I_w\subset 4I_v$. Let $c_v$ be the center of $R_v$ and let $\gamma_v = \kappa_{P_v,c_v}$, so that if $\ell = \frac{\delta_z(Q_v)}{2}$, then 
  $$rQ_v = \left\{(x,z)\mid x\in I_v, |z-\gamma_v(x)|\le  r^2\ell\right\}.$$
  Likewise, let $c_w$ be the center of $Q_w$ and $\gamma_w = \kappa_{P_w,c_w}$. We claim that $|\gamma_v(x)-\gamma_w(x)|\le 20\ell$ for $x\in I_w$.
  
  Let $q=(x_0,z_0)\in Q_w\cap 3Q_v$.
  Since $\gamma_v$ and $\kappa_{P_v, q}$ are parallel, we have $\gamma_v(x) = \kappa_{P_v, q}(x) + d$ for all $x$, where 
  $d=\gamma_v(x_0) - z_0$. Since $q\in 3Q_v$, we have $|d|\le 9\ell$. 
  
  Similarly, $\gamma_w(x) = \kappa_{P_w, q}(x) + d',$ and since $q\in Q_w$ and $\delta_z(Q_w)\le 4\delta_z(Q_v)$,
  $$|d'| = |\gamma_w(x_0) - z_0|\le \delta_z(Q_w)\le 8\ell.$$
  
  Therefore, by Lemma~\ref{lem:two-close-planes}, for all $x\in I_w$,
  \begin{align*}
    |\gamma_v(x)-\gamma_w(x)| &\le |d| + |d'| + |\kappa_{P_w, q}(x) - \kappa_{P_v, q}(x)| \\
    & \le 17\ell + |\kappa_{P_w, q}(x) - \kappa_{\Gamma, q}(x)| + |\kappa_{\Gamma, q}(x)- \kappa_{P_v, q}(x)|\\
    & \le 17\ell + \frac{\delta_z(Q_w)}{8} + \frac{\delta_z(Q_v)}{8} \le 20\ell.
  \end{align*}
  
  Let $p=(x, 0, z)\in Q_w$. Then $x\in I_w$ and $|z-\gamma_w(x)|\le \delta_z(Q_w)\le 8\ell$, so $|z - \gamma_v(x)| \le 28\ell$. Therefore $p\in 10Q_v$ and thus $Q_w\subset 10Q_v$.
\end{proof}

\bibliographystyle{alpha}
\bibliography{beta4}

\end{document}